\documentclass[reqno, 12pt]{amsart}%
\usepackage{amsmath}
\usepackage{amssymb}
\usepackage{amsfonts}
\usepackage{graphicx}%
\providecommand{\U}[1]{\protect\rule{.1in}{.1in}}
\newtheorem{theorem}{Theorem}
\theoremstyle{plain}

\newtheorem{corollary}{Corollary}

\newtheorem{definition}{Definition}

\newtheorem{lemma}{Lemma}

\newtheorem{proposition}{Proposition}
\newtheorem{remark}{Remark}

\DeclareMathOperator{\Div}{div}
\numberwithin{equation}{section}
\numberwithin{theorem}{section}
\numberwithin{proposition}{section}
\numberwithin{remark}{section}
\numberwithin{definition}{section}
\numberwithin{lemma}{section}
\numberwithin{corollary}{section}
\numberwithin{example}{section}
\numberwithin{claim}{section}
\begin{document}
\title[Homogenization of Linear Boltzmann Equations]{Homogenization of Linear Boltzmann Equations in the Context of Algebras with
Mean Value}
\author{P. Fouegap}
\address{P. Fouegap, Department of Mathematics and Computer Science, University of
Dschang, P.O. Box 67, Dschang, Cameroon}
\email{minlefack@gmail.com}
\author{R. Kenne B.}
\address{R. Kenne B., Department of Mathematics, University of Yaounde I, P.O. Box 812,
Yaounde, Cameroon}
\email{rodriguekenne@ymail.com}
\author{G. Nguetseng}
\address{G. Nguetseng, Department of Mathematics, University of Yaounde I, P.O. Box
812, Yaounde, Cameroon}
\email{nguetsengg@yahoo.fr}
\author{D. Dongo}
\address{D. Dongo, Department of Mathematics and Computer Science, University of
Dschang, P.O. Box 67, Dschang, Cameroon}
\email{dongodavid@yahoo.fr}
\author{J.L. Woukeng}
\address{J.L. Woukeng, Department of Mathematics and Computer Science, University of
Dschang, P.O. Box 67, Dschang, Cameroon}
\email{jwoukeng@yahoo.fr}
\date{May 2020}
\subjclass[2000]{35B40, 45M05, 82C70, 85A25}
\keywords{Deterministic homogenization, Boltzmann equations, algebras with mean value, sigma-convergence}

\begin{abstract}
The paper deals with the homogenization of a linear Boltzmann equation by the
means of the sigma-convergence method. Under a general deterministic
assumption on the coefficients of the equation, we prove that the density of
the particles converges to a solution of a drift-diffusion equation. To
achieve our goal, we use the Krein-Rutman theorem for locally convex spaces
together with the Fredholm alternative to solve the so-called corrector problem.

\end{abstract}
\maketitle


\section{Introduction and the main result\label{sec1}}

An important topic in multiscale analysis is the derivation of macroscopic
model equations from the microscopic ones arising from kinetic theory. One of
the most important kinetic equations is the Boltzmann equation, which roughly
reads as
\begin{equation}
\dfrac{\partial f}{\partial t}+v\cdot\nabla_{x}f=\mathcal{Q}f\text{.}%
\label{*1}%
\end{equation}
In (\ref{*1}), the left-hand side accounts for the total derivative that takes
into account the free streaming of particles, while the right-hand side is the
collision operator describing interactions between particles. The unknown
function $f$ refers to the distribution function and is physically interpreted
as the probability of the density of particles in a given volume. In this
work, we deal with a specific type of collision operator (see (\ref{Lf1}))
leading to the linear Boltzmann equation. The linear Boltzmann equation is a
kinetic model used in many different contexts. It appeared (for the first
time; see \cite{Lorentz}) on the motion of electrons in metals, and has been
since then used in various branches of mathematics and physics such as
radiative transfer \cite{Mihalas, Pomraning}, neutron transfer theory
\cite{Weinberg}. The macroscopic effects come to light when the time elapsing
between collisions is much smaller than the observation time scale. This
amounts to saying that, when the average distance between two successive
collisions is smaller than the given specimen length scale. In that case, it
therefore becomes interesting to seek the solution $f$ of (\ref{*1}) under the
form $f(t,x,v)=f_{\varepsilon}(\varepsilon t,x,v)$ where $0<\varepsilon<<1$ is
a small dimensionless parameter. So, considering the rescaled time variable
$\tau=\varepsilon t$, we see that (\ref{*1}) takes the form
\begin{equation}
\varepsilon\dfrac{\partial f_{\varepsilon}}{\partial\tau}(\tau,x,v)+v\cdot
\nabla_{x}f_{\varepsilon}(\tau,x,v)=\dfrac{1}{\varepsilon}\mathcal{Q}%
_{\varepsilon}f_{\varepsilon}(\tau,x,v).\label{*2}%
\end{equation}
Instead of (\ref{*2}) we rather consider the following initial value problem
\begin{equation}
\left\{
\begin{array}
[c]{l}%
\varepsilon\dfrac{\partial f_{\varepsilon}}{\partial t}+a(v)\cdot\nabla
_{x}f_{\varepsilon}=\dfrac{1}{\varepsilon}\mathcal{Q}_{\varepsilon
}f_{\varepsilon}\text{ in }(0,T)\times\mathbb{R}_{x}^{d}\times V\\
\\
f_{\varepsilon}(0,x,v)=f^{0}(x,v)\text{ in }\mathbb{R}_{x}^{d}\times V.
\end{array}
\right.  \label{EH0}%
\end{equation}
Such a problem naturally arises when modeling the behaviour of a cloud of
particles. The unknown function $f_{\varepsilon}(t,x,v)\geq0$ can be
interpreted as the density of particles occupying at time $t$, the position
$x$ with a physical state described by the variable $v$. As usually, $v$ is
the translation velocity of the particle and lies to the $d$-dimension space
$V\subset\mathbb{R}_{v}^{d}$ ($\mathbb{R}_{v}^{d}$ is the numerical space
$\mathbb{R}^{d}$ of generic variable $v$, integer $d\geq1$). The set $V$ is
endowed with a measure $d\mu$ whose crucial properties will be specified
later. The left-hand side of the first equation in (\ref{EH0}) describes the
transport of particles with the velocity field $a:V\rightarrow\mathbb{R}^{d}$,
while the right hand side takes into account the interactions that the
particles may undergo which crossing device. Coming back to (\ref{EH0}), the
$\varepsilon$ in front of the time derivative is related to the long time
scaling, while the $1/\varepsilon$ in front of collision operator means that
particles undergo with more and more interactions \cite{Goudon-Mellet}. Those
interactions modify the physical state of the particles, and are localized in
time and space. They can be described by the following integral operator:
\begin{equation}
\mathcal{Q}_{\varepsilon}f_{\varepsilon}(t,x,v))=\int_{V}\sigma^{\varepsilon
}(x,v,w)f_{\varepsilon}(t,x,w)d\mu(w)-\Sigma^{\varepsilon}(x,v)f_{\varepsilon
}(t,x,v),\label{Lf1}%
\end{equation}
where $\sigma^{\varepsilon}(x,v,w)=\sigma(x,\frac{x}{\varepsilon},v,w)$ (resp.
$\Sigma^{\varepsilon}(x,v)=\Sigma(x,\frac{x}{\varepsilon},v)$) is a
nonnegative function defined a.e. on $\mathbb{R}_{x}^{d}\times\mathbb{R}%
_{y}^{d}\times V\times V$ (resp. on $\mathbb{R}_{x}^{d}\times\mathbb{R}%
_{y}^{d}\times V$). The function $\sigma$ is called transition or scattering
rate, while $\Sigma$ is the absorption rate. In the case of this work, they
are both given. Assuming that the total density is conserved, that is
\begin{equation}
\int_{\mathbb{R}_{x}^{d}\times V}f^{\varepsilon}(t,x,v)d\mu(v)dx=\int
_{\mathbb{R}_{x}^{d}\times V}f^{0}(x,v)d\mu(v)dx,\label{todencon}%
\end{equation}
we are led to the following relation
\begin{equation}
\Sigma^{\varepsilon}(x,v)=\int_{V}\sigma^{\varepsilon}(x,w,v)d\mu
(w).\label{tdnc}%
\end{equation}
Moreover, we also assume that $\sigma$ (see e.g. \cite{BARDOS-BER-GOLSE,
Bardos}) satisfies the semi-detailed balance condition
\begin{equation}
\int_{V}\sigma^{\varepsilon}(x,v,w)d\mu(w)=\int_{V}\sigma^{\varepsilon
}(x,w,v)d\mu(w).\label{sdbc}%
\end{equation}

As in \cite{Goudon-Mellet}, we state some important hypotheses on the measure
space $(V,d\mu)$ and the velocity field $a(v)$. We suppose that
\begin{equation}
\left\{
\begin{array}
[c]{l}%
V\text{ is a compact subset of }\mathbb{R}^{d}\text{ and the measure }%
\mu\text{ satisfies }\mu(V)<\infty.\\
\text{The velocity field }a:V\rightarrow\mathbb{R}^{d}\text{ lies in
}W^{1,\infty}(V).\\
\text{There exists two constants }C,\gamma>0\text{ such that}\emph{\text{ }}\\
\mu(\left\{  v\in V:\left\vert a(v)\cdot\xi\right\vert \leq h\right\}  )\leq
Ch^{\gamma}\text{ for all }\xi\in S^{d-1},h>0
\end{array}
\right.  \label{H1}%
\end{equation}
where $S^{d-1}$ stands for the $d$-dimensional sphere in $\mathbb{R}^{d}$.
Next we also need the following assumption on $\sigma$:
\begin{equation}
\sigma\in\mathcal{B}(\mathbb{R}_{x}^{d}\times V\times V;B_{A}^{2,\infty
}(\mathbb{R}_{y}^{d})).\label{H2}%
\end{equation}
In (\ref{H2}), $A$ is a given algebra with mean value on $\mathbb{R}^{d}$
(that is, a closed subalgebra of the $\mathcal{C}^{\ast}$-Banach algebra of
bounded uniformly continuous real-valued functions on $\mathbb{R}^{d}$ that
contains the constants, is close under complex conjugation, is translation
invariant and is such that each of its elements $u$ possesses a mean value
$M(u)=\mathchoice {{\setbox0=\hbox{$\displaystyle{\textstyle
-}{\int}$ } \vcenter{\hbox{$\textstyle -$
}}\kern-.6\wd0}}{{\setbox0=\hbox{$\textstyle{\scriptstyle -}{\int}$ }
\vcenter{\hbox{$\scriptstyle -$
}}\kern-.6\wd0}}{{\setbox0=\hbox{$\scriptstyle{\scriptscriptstyle -}{\int}$
} \vcenter{\hbox{$\scriptscriptstyle -$
}}\kern-.6\wd0}}{{\setbox0=\hbox{$\scriptscriptstyle{\scriptscriptstyle
-}{\int}$ } \vcenter{\hbox{$\scriptscriptstyle -$ }}\kern-.6\wd0}}\!\int
_{B_{R}}u(y)dy$ where we set once and for all
$\mathchoice {{\setbox0=\hbox{$\displaystyle{\textstyle
-}{\int}$ } \vcenter{\hbox{$\textstyle -$
}}\kern-.6\wd0}}{{\setbox0=\hbox{$\textstyle{\scriptstyle -}{\int}$ }
\vcenter{\hbox{$\scriptstyle -$
}}\kern-.6\wd0}}{{\setbox0=\hbox{$\scriptstyle{\scriptscriptstyle -}{\int}$
} \vcenter{\hbox{$\scriptscriptstyle -$
}}\kern-.6\wd0}}{{\setbox0=\hbox{$\scriptscriptstyle{\scriptscriptstyle
-}{\int}$ } \vcenter{\hbox{$\scriptscriptstyle -$ }}\kern-.6\wd0}}\!\int
_{S}=\frac{1}{\left\vert S\right\vert }\int_{S}$ for any measurable set
$S\subset\mathbb{R}^{d}$), $B_{A}^{2,\infty}(\mathbb{R}_{y}^{d})=B_{A}%
^{2}(\mathbb{R}_{y}^{d})\cap L^{\infty}(\mathbb{R}_{y}^{d})$ where $B_{A}%
^{2}(\mathbb{R}^{d})$ is the generalized Besicovitch space defined as the
closure of the algebra $A$ with respect to the seminorm $\left\Vert
u\right\Vert _{2}=(M(\left\vert u\right\vert ^{2}))^{\frac{1}{2}}$, and
$\mathcal{B}(\mathbb{R}^{d}\times V\times V;B_{A}^{2,\infty}(\mathbb{R}%
_{y}^{d}))$ denotes the space of bounded continuous functions from
$\mathbb{R}^{d}\times V\times V$ into $B_{A}^{2,\infty}(\mathbb{R}_{y}^{d})$.
The above spaces will be precise in the next section. However let us pay
special attention to assumption (\ref{H2}). When accounting of the specific
properties of the medium in which the collisions occur, the analysis of our
model becomes more involved and necessitates special attention. These
properties are naturally included in the behaviour of the scattering rate
function $\sigma$, and they influence the overall behaviour of the density
function $f_{\varepsilon}$. They depend on the way the microstructures are
distributed in the heterogeneous medium (where the collisions occur). For
example, we could assume that the medium is made of microstructures that are
either uniformly distributed inside or almost uniformly distributed, or assume
another kind of deterministic distribution. This leads to a function
$y\mapsto\sigma(x,y,v,w)$ which is assumed to be either periodic (with respect
to $y$) or almost periodic, or even asymptotic almost periodic. All these
properties are included in assumption (\ref{H2}).

Finally we suppose that the initial distribution function $f^{0}$, satisfies
\begin{equation}
f^{0}(x,v)>0\text{ and }\int_{\mathbb{R}^{d}\times V}\left\vert f^{0}%
\right\vert ^{2}(x,v)dxd\mu(v)\leq C_{0}<\infty. \label{H3}%
\end{equation}

Under hypotheses (\ref{H1})-(\ref{H3}), the Cauchy problem (\ref{EH0}) has
(for each fixed $\varepsilon>0$) a unique weak solution $f_{\varepsilon}%
\in\mathcal{C}([0,T];L^{2}(\mathbb{R}^{d}\times V))$ (see e.g.
\cite{BARDOS-BER-GOLSE, Bardos} for the proof), which further satisfies the
following estimate
\begin{equation}
\left\Vert f_{\varepsilon}\right\Vert _{L^{\infty}([0,T];L^{2}(\mathbb{R}%
^{d}\times V))}\leq C \label{estimf}%
\end{equation}
where $C>0$ is independent of $\varepsilon>0$. The proof of estimate
(\ref{estimf}) follows from direct application of the energy method to the
system (\ref{EH0}), using the semi-detailed balance condition (\ref{sdbc}).

Defining the operator $P$ on $L^{2}(V;B_{A}^{2}(\mathbb{R}_{y}^{d}))$ by
\[
Pu=a(v)\cdot\nabla_{y}u-\mathcal{Q}u\ \ (u\in L^{2}(V;B_{A}^{2}(\mathbb{R}%
_{y}^{d})))
\]
where $\mathcal{Q}u(x,y,v)=\int_{V}\sigma(x,y,v,w)(u(y,w)-u(y,v))d\mu(w)$, we
know that there exists a unique $F(x,\cdot,\cdot)\in L^{2}(V;B_{A}%
^{2}(\mathbb{R}_{y}^{d}))$ such that
\begin{equation}
PF=0\text{, }\int_{V}M(F(x,\cdot,v))d\mu(v)=1\text{ and }F>0\text{ a.e.;}
\label{0.1}%
\end{equation}
see Proposition \ref{propkers}. We assume that the velocity field $a(v)$
satisfies the following vanishing flux condition:
\begin{equation}
\int_{V}M(a(v)F(x,\cdot,v))d\mu(v)=0. \label{vfc}%
\end{equation}
The following theorem is the main result of the work.

\begin{theorem}
\label{theosol}Let $A$ be an algebra with mean value on $\mathbb{R}^{d}$.
Assume \emph{(\ref{H2})} and \emph{(\ref{vfc})} hold. For each $\varepsilon
>0$, let $f_{\varepsilon}$ be the unique solution of \emph{(\ref{EH0})}. Then,
there exists $f_{0}+\mathcal{N}\in L^{2}(\mathbb{R}_{T}^{d}\times V;B_{A}%
^{2}(\mathbb{R}^{d})/\mathcal{N})$ such that the sequence $(f_{\varepsilon
})_{\varepsilon\in E}$ weakly sigma-converges in $L^{2}((0,T)\times
\mathbb{R}_{x}^{d}\times V)$ towards $f_{0}+\mathcal{N}$. Moreover $f_{0}$ has
the form $f_{0}(t,x,y,v)=F(x,y,v)\rho_{0}(t,x)$ where $F\in\mathcal{C}%
^{1}(\mathbb{R}_{x}^{d};L^{2}(V;B_{A}^{2}(\mathbb{R}_{y}^{d}))\cap\ker P)$ is
given by \emph{(\ref{0.1})} and $\rho_{0}$ is the unique solution of the
Cauchy problem
\begin{equation}
\left\{
\begin{array}
[c]{l}%
\frac{\partial\rho_{0}}{\partial t}-\Div_{x}\left(  D(x)^{T}\nabla_{x}\rho
_{0}+U(x)\rho_{0}\right)  =0\text{ in }(0,T)\times\mathbb{R}_{x}%
^{d}\ \ \ \ \ \ \ \ \ \ \ \ \ \\
\\
\rho_{0}(0,x)=\int_{V}f^{0}(x,v)d\mu(v)\text{ in }\mathbb{R}_{x}^{d}%
\end{array}
\right.  \label{solf0}%
\end{equation}
where
\[
D(x)=\int_{V}M\left(  \chi^{\ast}\otimes(a(v)F)\right)  d\mu(v)=\left(
\int_{V}M\left(  \chi_{i}^{\ast}a_{j}(v)F\right)  d\mu(v)\right)  _{1\leq
i,j\leq d}%
\]
with $(\chi^{\ast}\otimes a(v)F)=\left(  \chi_{i}^{\ast}a_{j}(v)F\right)
_{1\leq i,j\leq d}$,%
\[
U(x)=\int_{V}M\left(  \chi^{\ast}a(v)\cdot\nabla_{x}F\right)  d\mu(v)=\left(
\int_{V}M\left(  \chi_{i}^{\ast}a(v)\cdot\nabla_{x}F\right)  d\mu(v)\right)
_{1\leq i\leq d}%
\]
and where $\chi^{\ast}$ the unique solution in $\mathcal{C}_{0}^{\infty
}(\mathbb{R}^{d}\times V;B_{A}^{2}(\mathbb{R}_{y}^{d}))^{d}$ of the corrector
problem:
\begin{equation}
P^{\ast}\chi^{\ast}=-a(v)\text{ and }\int_{V}M(\chi^{\ast}(x,\cdot
,v))d\mu(v)=0,\label{defchi}%
\end{equation}
$P^{\ast}$ being the adjoint operator of $P$.
\end{theorem}

The equation (\ref{solf0}) models a drift-diffusion process. The diffusion
arises by the term $\Div_{x}(D(x)^{T}\nabla_{x}\rho_{0})$ which represents the
spatial redistribution of particles forced by their own kinetic energy, while
the drift, represented by the term $\Div_{x}(U(x)\rho_{0})$, models the
transport of particles forced by a given outer field or a field generated by
the particles in a self-consistent manner.

Theorem \ref{theosol} has been proved in \cite{Goudon-Mellet} in the periodic
setting, that is, when assuming that the scattering function $\sigma$ is
periodic. To achieve their goal in \cite{Goudon-Mellet}, the authors used the
Krein-Rutman theorem in Banach spaces (see \cite[Theorem VI.12, page
100]{brezis}). However, as the Besicovitch spaces $B_{A}^{p}(\mathbb{R}^{d})$
are not Banach spaces (but rather complete locally convex topological vector
spaces) we can not expect to use the initial version of that theorem as
presented in \cite{brezis}. Fortunately the initial version has been
generalized to locally convex spaces by Schaefer \cite{Schaefer}, and it is
the one that we use in this work, which allows one to relax the periodicity
assumption by considering a general deterministic assumption which includes
the periodic one and the almost periodic one as special cases. This falls
within the scope of deterministic homogenization theory which proceeds from
the juxtaposition of the algebras with mean value and the sigma-convergence
method. To the best of our knowledge, this is the first time that such a
result is obtained beyond the periodic setting.

The homogenization of Boltzmann equation remains an active field of study. We
may cite \cite{HUTRIDURGA, BARDOS-BER-GOLSE, Bardos, ABTAY, BerCagGolse,
Goudon-Mellet, MASTAY}. In \cite{ABTAY, MASTAY}, the authors study the case of
Diffusion limit of a semiconductor Boltzmann-Poisson by using the Hilbert
formal expansion, in the case of periodic oscillations. In \cite{HUTRIDURGA,
BerCagGolse, Goudon-Mellet} the authors investigate the asymptotic behaviour
of a linear Boltzmann equation by the two-scale convergence method. In all
these works, hypotheses were stated in the periodic setting. Here, we use the
sigma convergence method to handle the more general case including as special
settings, the almost periodic homogenization, the weak almost periodic one,
and others.

The rest of the paper is organized as follows. In Section \ref{sec3}, we
present some fundamental results about the concept of sigma-convergence.
Section \ref{sec4} deals with the corrector result. Section \ref{sec5} is
devoted to the proof of Theorem \ref{theosol}. Finally in Section \ref{sec6}
we provide some concrete examples in which Theorem \ref{theosol} applies.

\section{Algebras with mean value and Sigma convergence\label{sec3}}

In this section we recall the main properties and some basic facts about the
concept of sigma-convergence. We refer the reader to \cite{22, 28, Woukeng2}
for the details regarding most of the results of this section.

In the sequel of this work, we will often set $\mathbb{R}_{T}^{d}%
=(0,T)\times\mathbb{R}_{x}^{d}$.

\subsection{Algebra with mean value\label{subsec3.1}}

Let $A$ be an algebra with mean value (algebra wmv, for short) on
$\mathbb{R}^{d}$, that is, a closed subalgebra of the Banach algebra
$\mathrm{BUC}(\mathbb{R}^{d})$ (of bounded uniformly continuous real-valued
functions on $\mathbb{R}^{d}$) that contains the constants, is close under
complex conjugation ($\overline{u}\in A$ whenever $u\in A$), is translation
invariant ($\tau_{a}u=u(\cdot+a)\in A$ for any $u\in A$ and $a\in
\mathbb{R}^{d}$) and is such that any of its elements possesses a mean value
in the following sense: for every $u\in A$,%

\begin{equation}
M(u)=\lim_{R\rightarrow\infty}%
\mathchoice {{\setbox0=\hbox{$\displaystyle{\textstyle
-}{\int}$ } \vcenter{\hbox{$\textstyle -$
}}\kern-.6\wd0}}{{\setbox0=\hbox{$\textstyle{\scriptstyle -}{\int}$ } \vcenter{\hbox{$\scriptstyle -$
}}\kern-.6\wd0}}{{\setbox0=\hbox{$\scriptstyle{\scriptscriptstyle -}{\int}$
} \vcenter{\hbox{$\scriptscriptstyle -$
}}\kern-.6\wd0}}{{\setbox0=\hbox{$\scriptscriptstyle{\scriptscriptstyle
-}{\int}$ } \vcenter{\hbox{$\scriptscriptstyle -$ }}\kern-.6\wd0}}\!\int
_{B_{R}}{u(y)}dy \label{eq4.1}%
\end{equation}
where $B_{R}$ stands for the open ball in $\mathbb{R}^{d}$ of radius $R$
centered at the origin and
$\mathchoice {{\setbox0=\hbox{$\displaystyle{\textstyle
-}{\int}$ } \vcenter{\hbox{$\textstyle -$
}}\kern-.6\wd0}}{{\setbox0=\hbox{$\textstyle{\scriptstyle -}{\int}$ } \vcenter{\hbox{$\scriptstyle -$
}}\kern-.6\wd0}}{{\setbox0=\hbox{$\scriptstyle{\scriptscriptstyle -}{\int}$
} \vcenter{\hbox{$\scriptscriptstyle -$
}}\kern-.6\wd0}}{{\setbox0=\hbox{$\scriptscriptstyle{\scriptscriptstyle
-}{\int}$ } \vcenter{\hbox{$\scriptscriptstyle -$ }}\kern-.6\wd0}}\!\int
_{B_{R}}=\frac{1}{|B_{R}|}\int_{B_{R}}$.

Let $u\in\mathrm{BUC}(\mathbb{R}^{d})$ and assume that $M(u)$ exists. Then
defining the sequence $(u^{\varepsilon})_{\varepsilon>0}\subset\mathrm{BUC}%
(\mathbb{R}^{d})$ by $u^{\varepsilon}(x)=u(\frac{x}{\varepsilon})$ for
$x\in\mathbb{R}^{d}$, we have
\[
u^{\varepsilon}\rightarrow M(u)\text{ in }L^{\infty}(\mathbb{R}^{d}%
)\text{-weak}\ast\text{ as }\varepsilon\rightarrow0.
\]
This is an easy consequence of the fact that the set of finite linear
combinations of the characteristic functions of open balls in $\mathbb{R}^{d}$
is dense in $L^{1}(\mathbb{R}^{d})$.

Let $A$ be an algebra with mean value. Define the space $A^{\infty}$ by
\[
A^{\infty}=\left\{  u\in A:D_{y}^{\alpha}u\in A\text{ for every }%
\alpha=(\alpha_{1},...,\alpha_{d})\in\mathbb{N}^{d}\right\}  .
\]
Then endowed with the family of norms $\left\Vert \left\vert \cdot\right\vert
\right\Vert _{m}$ defined by $\left\Vert \left\vert u\right\vert \right\Vert
_{m}=\sup_{|\alpha|\leq m}\sup_{y\in\mathbb{R}^{d}}|D_{y}^{\alpha}u|$ where
$D_{y}^{\alpha}=\frac{\partial^{|\alpha|}}{\partial y_{1}^{\alpha_{1}%
}...\partial y_{d}^{\alpha_{d}}}$, $A^{\infty}$ is a Fr\'{e}chet space.

In order to define the generalized Besicovitch space, we first need to define
the Marcinkiewicz space $\mathfrak{M}^{p}(\mathbb{R}^{d})$ ($1\leq p<\infty$),
which is the space of functions $u\in{L}_{loc}^{p}(\mathbb{R}^{d})$ satisfying
$\limsup_{R\rightarrow\infty}%
\mathchoice {{\setbox0=\hbox{$\displaystyle{\textstyle
-}{\int}$ } \vcenter{\hbox{$\textstyle -$
}}\kern-.6\wd0}}{{\setbox0=\hbox{$\textstyle{\scriptstyle -}{\int}$ } \vcenter{\hbox{$\scriptstyle -$
}}\kern-.6\wd0}}{{\setbox0=\hbox{$\scriptstyle{\scriptscriptstyle -}{\int}$
} \vcenter{\hbox{$\scriptscriptstyle -$
}}\kern-.6\wd0}}{{\setbox0=\hbox{$\scriptscriptstyle{\scriptscriptstyle
-}{\int}$ } \vcenter{\hbox{$\scriptscriptstyle -$ }}\kern-.6\wd0}}\!\int
_{B_{R}}\left\vert {u(y)}\right\vert ^{p}dy<\infty$. Endowed with the
seminorm
\[
\left\Vert u\right\Vert _{p}=\limsup_{R\rightarrow\infty}\left(
\mathchoice {{\setbox0=\hbox{$\displaystyle{\textstyle
-}{\int}$ } \vcenter{\hbox{$\textstyle -$
}}\kern-.6\wd0}}{{\setbox0=\hbox{$\textstyle{\scriptstyle -}{\int}$ } \vcenter{\hbox{$\scriptstyle -$
}}\kern-.6\wd0}}{{\setbox0=\hbox{$\scriptstyle{\scriptscriptstyle -}{\int}$
} \vcenter{\hbox{$\scriptscriptstyle -$
}}\kern-.6\wd0}}{{\setbox0=\hbox{$\scriptscriptstyle{\scriptscriptstyle
-}{\int}$ } \vcenter{\hbox{$\scriptscriptstyle -$ }}\kern-.6\wd0}}\!\int
_{B_{R}}\left\vert {u(y)}\right\vert ^{p}dy\right)  ^{\frac{1}{p}},
\]
$\mathfrak{M}^{p}(\mathbb{R}^{d})$ is a complete seminormed space. Next we
define the generalized Besicovitch space $B_{A}^{p}(\mathbb{R}^{d})$ ($1\leq
p<\infty$) associated to the algebra with mean value $A$ as the closure in
$\mathfrak{M}^{p}(\mathbb{R}^{d})$ of $A$ with respect to $\left\Vert
\cdot\right\Vert _{p}$. It is easy to see that for $f\in A$ and $0<p<\infty$,
$\left\vert f\right\vert ^{p}\in A$, so that
\begin{equation}
\left\Vert f\right\Vert _{p}=\left(  \lim_{R\rightarrow\infty}%
\mathchoice {{\setbox0=\hbox{$\displaystyle{\textstyle
-}{\int}$ } \vcenter{\hbox{$\textstyle -$
}}\kern-.6\wd0}}{{\setbox0=\hbox{$\textstyle{\scriptstyle -}{\int}$ } \vcenter{\hbox{$\scriptstyle -$
}}\kern-.6\wd0}}{{\setbox0=\hbox{$\scriptstyle{\scriptscriptstyle -}{\int}$
} \vcenter{\hbox{$\scriptscriptstyle -$
}}\kern-.6\wd0}}{{\setbox0=\hbox{$\scriptscriptstyle{\scriptscriptstyle
-}{\int}$ } \vcenter{\hbox{$\scriptscriptstyle -$ }}\kern-.6\wd0}}\!\int
_{B_{R}}{\left\vert f(y)\right\vert ^{p}}\right)  ^{\frac{1}{p}}\equiv\left(
M(\left\vert f\right\vert ^{p})\right)  ^{\frac{1}{p}}. \label{3.0}%
\end{equation}
The equality (\ref{3.0}) extends by continuity to any $f\in B_{A}%
^{p}(\mathbb{R}^{d})$. Equipped with the seminorm (\ref{3.0}), $B_{A}%
^{p}(\mathbb{R}^{d})$ is a complete seminormed space. We refer the reader to
\cite{MSANGO, 28, Woukeng2} for further details about these spaces. Namely,
the following holds true:

\begin{enumerate}
\item[(1)] The space $\mathcal{B}_{A}^{p}(\mathbb{R}^{d})=B_{A}^{p}%
(\mathbb{R}^{d})/\mathcal{N}$, (where $\mathcal{N}=\{u\in B_{A}^{p}%
(\mathbb{R}^{d}):\left\Vert u\right\Vert _{p}=0\}$) is a Banach space under
the norm $\left\Vert u+\mathcal{N}\right\Vert _{p}=\left\Vert u\right\Vert
_{p}$ for $u\in B_{A}^{p}(\mathbb{R}^{d})$.

\item[(2)] The mean value $M:A\rightarrow\mathbb{R}$ extends by continuity to
a continuous linear mapping (still denoted by $M$) on $B_{A}^{p}%
(\mathbb{R}^{d})$. Furthermore, considered as defined on $B_{A}^{p}%
(\mathbb{R}^{d})$, $M$ extends in a natural way to $\mathcal{B}_{A}%
^{p}(\mathbb{R}^{d})$ as follows: for $u=v+\mathcal{N}\in\mathcal{B}_{A}%
^{p}(\mathbb{R}^{d})$, we set $M(u):=M(v)$; this is well defined since
$M(v)=0$ for any $v\in\mathcal{N}$.
\end{enumerate}

In the current work, we will deal with the concept of \textit{ergodic}
algebras with mean value. A function $u\in\mathcal{B}_{A}^{1}(\mathbb{R}^{d})$
is said to be \textit{invariant} if for any $y\in\mathbb{R}^{d}$, $\left\Vert
u(\cdot+y)-u\right\Vert _{1}=0$. This being so, an algebra with mean value $A$
is ergodic if every invariant function $u$ is constant in $\mathcal{B}_{A}%
^{1}(\mathbb{R}^{d})$, i.e. if $\left\Vert u(\cdot+y)-u\right\Vert _{1}=0$ for
any $y\in\mathbb{R}^{d}$, then $\left\Vert u-c\right\Vert _{1}=0$ where $c$ is
a constant. We assume that all the algebras with mean value used in the sequel
are ergodic.

\begin{remark}
\label{R1}\emph{In the sequel we assume that all algebras wmv consist of
real-valued functions. This does not make any restriction on the preceding
results or on the forthcoming ones. We will denote the Gelfand transform
}$\mathcal{G}(u)$\emph{ of a function }$u\in A$\emph{ by }$\widehat{u}%
$\emph{.}
\end{remark}

\subsection{The Sigma-convergence\label{subsec3.2}}

In what follows, the notations are those of the preceding subsections.

Let $A$ be an algebra wmv on $\mathbb{R}_{y}^{d}$. We know that $A$ is a
subalgebra of the $\mathcal{C}^{\ast}$-algebra of bounded uniformly continuous
functions $\mathrm{BUC}(\mathbb{R}_{y}^{d})$. The generic element of
$\mathbb{R}_{T}^{d}$ is denoted by $(t,x)$ while any function in $A$ is of
variable $y\in\mathbb{R}_{y}^{d}$. For a function $u\in L^{p}(\mathbb{R}%
_{T}^{d}\times V;B_{A}^{p}(\mathbb{R}^{d}))$ we denote by $u(t,x,\cdot,v)$
(for a.e. $(t,x,v)\in\mathbb{R}_{T}^{d}\times V$) the function defined by
\[
u(t,x,\cdot,v)(y)=u(t,x,y,v)\text{ for a.e. }y\in\mathbb{R}^{d}.
\]
Then $u(t,x,\cdot,v)\in B_{A}^{p}(\mathbb{R}^{d})$, so that the mean value of
$u(t,x,\cdot,v)$ is defined accordingly.

\begin{definition}
\label{d2}\emph{A sequence} $(u_{\varepsilon})_{\varepsilon>0}\subset
L^{p}(\mathbb{R}_{T}^{d}\times V)$ $(1\leq p<\infty)$ \emph{is weakly}
sigma-convergent\emph{ in }$L^{p}(\mathbb{R}_{T}^{d}\times V)$ \emph{to some
function } $u_{0}\in L^{p}(\mathbb{R}_{T}^{d}\times V;\mathcal{B}_{A}%
^{p}(\mathbb{R}^{d}))$ \emph{if as }$\varepsilon\rightarrow0$,\emph{ we have
}
\begin{equation}
\iint_{\mathbb{R}_{T}^{d}\times V}u_{\varepsilon}(t,x,v)\varphi^{\varepsilon
}(t,x,v)d\mu(v)dxdt\rightarrow\iint_{\mathbb{R}_{T}^{d}\times V}%
M(u_{0}(t,x,\cdot,v)\varphi(t,x,\cdot,v))d\mu(v)dxdt \label{26}%
\end{equation}
\emph{for every }$\varphi\in L^{p^{\prime}}(\mathbb{R}_{T}^{d}\times
V;A)$\emph{ }$(\frac{1}{p^{\prime}}=1-\frac{1}{p})$\emph{, where }%
$\varphi^{\varepsilon}(t,x,v)=\varphi(t,x,\frac{x}{\varepsilon},v)$. \emph{We
express this by writing }$u_{\varepsilon}\rightarrow u_{0}$ \emph{in }%
$L^{p}(\mathbb{R}_{T}^{d}\times V)$\text{-weak }$\Sigma$.
\end{definition}

In the above definition, if $A=\mathcal{C}_{per}(Y)$ is the algebra of
continuous periodic functions on $Y$ with $Y=(0,1)^{d}$, then (\ref{26}) reads
as
\[
\iint_{\mathbb{R}_{T}^{d}\times V}u_{\varepsilon}(t,x,v)\varphi^{\varepsilon
}(t,x,v)d\mu(v)dxdt\rightarrow\iint_{\mathbb{R}_{T}^{d}\times V\times Y}%
u_{0}(t,x,y,v)\varphi(t,x,y,v)dyd\mu(v)dxdt
\]
where $u_{0}\in L^{p}(\mathbb{R}_{T}^{d}\times V\times Y)$.

\begin{remark}
\label{r2}\emph{One may show as in \cite{22} that, for any }$f\in
L^{p}(\mathbb{R}_{T}^{d}\times V;A)$\emph{, the sequence }$(f^{\varepsilon
})_{\varepsilon>0}$\emph{ weakly }$\Sigma$\emph{-converges towards
}$f+\mathcal{N}$\emph{ in }$L^{p}(\mathbb{R}_{T}^{d}\times V)$\emph{. It can
be shown (see \cite[Corollary 4.1]{22}) that property (\ref{26}) still holds
for }$\varphi\in L^{p^{\prime}}(\mathbb{R}^{d};\mathcal{C}(\mathbb{R}_{T}%
^{d};B_{A}^{p^{\prime},\infty}(\mathbb{R}_{y}^{d})))$\emph{ where }%
$B_{A}^{p^{\prime},\infty}=B_{A}^{p^{\prime}}\cap L^{\infty}(\mathbb{R}%
_{y}^{d})$\emph{ is endowed with the }$L^{\infty}(\mathbb{R}_{y}^{d}%
)$\emph{-norm.}
\end{remark}

In the sequel, the letter $E$ will always denote any ordinary sequence
$(\varepsilon_{n})_{n\in\mathbb{N}}$ of positive real numbers satisfying:
$0<\varepsilon_{n}\leq1$ and $\varepsilon_{n}\rightarrow0$ when $n\rightarrow
\infty$.

The following result and its proof are simple adaptation of its counterpart in
\cite{19, 22, 28}.

\begin{theorem}
\label{Theohom1}Let $1<p<\infty$ and let $(u_{\varepsilon})_{\varepsilon\in
E}\subset L^{p}(\mathbb{R}_{T}^{d}\times V)$ be a bounded sequence. Then there
exist a subsequence $E^{\prime}$ of $E$ and a function $u\in L^{p}%
(\mathbb{R}_{T}^{d}\times V;\mathcal{B}_{A}^{p}(\mathbb{R}^{d}))$ such that
the sequence $(u_{\varepsilon})_{\varepsilon\in E^{\prime}}$ weakly $\Sigma
$-converges in $L^{p}(\mathbb{R}_{T}^{d}\times V)$ to $u$.
\end{theorem}

\section{Corrector results\label{sec4}}

In this section we aim at stating and proving an important result based on
Krein-Rutman theorem \cite[Theorem VI.13, p. 100]{brezis}. It is very
important to note that the Krein-Rutman theorem was first stated and proved by
Krein and Rutman \cite{KR}. A more general version of that theorem appeared in
Bonsall \cite{Bo} and was extended to locally convex topological spaces by
Schaefer \cite{Schaefer}. Since the space $L^{2}(V;B_{A}^{2}(\mathbb{R}^{d}))$
is rather a complete locally convex topological vector space, we shall follow
in this subsection the lines of \cite{Schaefer}. However, we can observe that
the results obtained here can be extended to the original Krein-Rutman setting
of Banach spaces by considering the Hilbert space $L^{2}(V;\mathcal{B}_{A}%
^{2}(\mathbb{R}^{d}))$ which is the Banach counterpart of $L^{2}(V;B_{A}%
^{2}(\mathbb{R}^{d}))$.

To start with, we consider the unbounded operator $P$ defined on
$L^{2}(V;B_{A}^{2}(\mathbb{R}^{d}))$ by
\begin{equation}
Pu=a(v)\cdot\nabla_{y}u-\mathcal{Q}u \label{eq1}%
\end{equation}
with domain
\[
D(P)=\{u\in L^{2}(V;B_{A}^{2}(\mathbb{R}_{y}^{d})):Pu\in L^{2}(V;B_{A}%
^{2}(\mathbb{R}_{y}^{d}))\}
\]
where $\mathcal{Q}$ is the operator defined by
\begin{equation}
\mathcal{Q}f(x,y,v)=\int_{V}\sigma(x,y,v,w)(f(y,w)-f(y,v))d\mu(w),\ f\in
L^{2}(V;B_{A}^{2}(\mathbb{R}_{y}^{d})) \label{OpQ}%
\end{equation}
and the equality (\ref{eq1}) above holds in the weak sense in $L^{2}%
(V;B_{A}^{2}(\mathbb{R}_{y}^{d}))$, that is
\begin{equation}
(Pu,\phi)=\int_{V}M(-ua(v)\cdot\nabla_{y}\phi-\phi\mathcal{Q}u)d\mu(v)\text{
for all }\phi\in\mathcal{C}^{\infty}(V;A^{\infty}). \label{eq2}%
\end{equation}
Using the semi-detailed balance condition (\ref{sdbc}) we easily show that
\begin{equation}
\int_{V}(\phi\mathcal{Q}u)d\mu(v)=\int_{V}(u\mathcal{Q}^{\ast}\phi)d\mu(v)
\label{intvuLw}%
\end{equation}
where $\mathcal{Q}^{\ast}$ is the adjoint collision operator of $\mathcal{Q}$
defined by
\begin{equation}
\mathcal{Q}^{\ast}\phi(x,y,v)=\int_{V}\sigma(x,y,w,v)(\phi(y,w)-\phi
(y,v))d\mu(w), \label{adjOpQ}%
\end{equation}
so that (\ref{eq2}) is equivalent to
\begin{equation}
(Pu,\phi)=\int_{V}M(u(-a(v)\cdot\nabla_{y}\phi-\mathcal{Q}^{\ast}\phi
))d\mu(v)\text{ for all }\phi\in\mathcal{C}^{\infty}(V;A^{\infty}).
\label{eq3}%
\end{equation}
We infer from (\ref{eq3}) that the adjoint $P^{\ast}$ of $P$ is well defined
by
\begin{equation}
P^{\ast}\phi=-a(v)\cdot\nabla_{y}\phi-\mathcal{Q}^{\ast}\phi. \label{eq3'}%
\end{equation}
Next, if $u\in L^{2}(V;\mathcal{N})\cap D(P)$ ($\mathcal{N}=\{f\in B_{A}%
^{2}(\mathbb{R}_{y}^{d})):M(\left\vert f\right\vert ^{2})=0\}$) then $Pu=0$,
that is
\begin{equation}
(Pu,\phi)=0\text{ for all }\phi\in\mathcal{C}^{\infty}(V;A^{\infty}).
\label{eq4}%
\end{equation}
This shows that $P$ can be extended on the Banach space $L^{2}(V;\mathcal{B}%
_{A}^{2}(\mathbb{R}_{y}^{d}))$ by
\begin{equation}
\mathcal{P}\widetilde{u}=\widetilde{Pu}\text{ for }=u+\mathcal{N}\text{ with
}u\in L^{2}(V;B_{A}^{2}(\mathbb{R}_{y}^{d})) \label{eq5}%
\end{equation}
where $\widetilde{Pu}=Pu+\mathcal{N}$. Indeed in view of (\ref{eq4}), if
$\widetilde{u}=\widetilde{w}$ then $u-w\in\mathcal{N}$, so that $P(u-w)=0$,
that is, $\mathcal{P}\widetilde{u}=\mathcal{P}\widetilde{w}$. The domain of
$\mathcal{P}$ is naturally defined by
\[
D(\mathcal{P})=\{\widetilde{u}\in L^{2}(V;\mathcal{B}_{A}^{2}(\mathbb{R}%
_{y}^{d})):\mathcal{P}\widetilde{u}\in L^{2}(V;\mathcal{B}_{A}^{2}%
(\mathbb{R}_{y}^{d}))\}.
\]
With the above definitions of $P$ and $\mathcal{P}$, in order to study the
properties of $\mathcal{P}$, it will be sufficient to study those of $P$ based
on the framework developed in \cite{Schaefer}.

The Proposition \ref{propkers} below is a generalization of its counterpart in
\cite[Section 3]{Goudon-Mellet} where the periodic setting is treated by
considering the usual Krein-Rutman theorem using Banach spaces theory. Here we
consider the general deterministic framework which does not fall within the
scope of the usual Krein-Rutman theorem in Banach spaces as in \cite[Section
3]{Goudon-Mellet}, but rather within the scope of its generalization to
locally convex spaces. However as seen in \cite{Schaefer}, the assumptions in
applying the results from \cite{Schaefer} are the same as those from the
Banach setting. We start with the description of the kernel of the operators
$P$ and $P^{\ast}$ (the estimates obtained here are uniform with respect to
the spatial macroscopic variable $x\in\mathbb{R}^{d}$). The overall aim of
this subsection is to solve the auxiliary problems (\ref{eq6}) and (\ref{eq7}) below:

\begin{itemize}
\item the corrector problem:
\begin{equation}
\text{For }g\in L^{2}(V;B_{A}^{2}(\mathbb{R}_{y}^{d}))\text{, look for }f\in
D(P)\text{ such that }Pf=g\text{ } \label{eq6}%
\end{equation}
and

\item the dual corrector problem:
\begin{equation}
\text{For }\varphi\in L^{2}(V;B_{A}^{2}(\mathbb{R}_{y}^{d}))\text{, look for
}\phi\in D(P^{\ast})\text{ such that }P^{\ast}\phi=\varphi\text{.} \label{eq7}%
\end{equation}
We recall that equalities (\ref{eq6}) and (\ref{eq7}) hold algebraically and
so, also hold in the usual sense of distributions in $\mathbb{R}^{d}\times V$.
It will also be seen below that they will hold in the weak sense in
$L^{2}(V;B_{A}^{2}(\mathbb{R}_{y}^{d}))$ (see especially the proof of part (i)
of the proposition below).
\end{itemize}

\begin{proposition}
\label{propkers}Suppose \emph{(\ref{H1})} holds true and that $\sigma
=\sigma(x,\cdot,\cdot,\cdot)\in L^{\infty}(V\times V;B_{A}^{2,\infty
}(\mathbb{R}_{y}^{d}))$. Then the following assertions hold:

\begin{itemize}
\item[(i)] There exists a unique function $F\in L^{2}(V;B_{A}^{2}%
(\mathbb{R}_{y}^{d}))$ satisfying
\begin{equation}
PF=0,\ \int_{V}M(F)d\mu(v)=1\text{ and}\ F>0\text{ a.e. in }\mathbb{R}_{y}%
^{d}\times V. \label{KerP}%
\end{equation}
Furthermore $F$ is the unique solution of the variational equation
\begin{equation}
\int_{V}M(F(a(v)\cdot\nabla_{y}\phi(\cdot,v)+\mathcal{Q}^{\ast}\phi
(\cdot,v)))d\mu(v)=0\text{ for all }\phi\in C^{\infty}(V;A^{\infty}).
\label{eq17}%
\end{equation}
Similarly, we have $\ker P^{\ast}=\mathrm{span}\{\chi_{\mathbb{R}_{y}%
^{d}\times V}\}$, where $\chi_{\mathbb{R}_{y}^{d}\times V}$ stands for the
characteristic function of $\mathbb{R}_{y}^{d}\times V$.

\item[(ii)] Let $g\in L^{2}(V;B_{A}^{2}(\mathbb{R}_{y}^{d}))$. The equation
\emph{(\ref{eq6})} admits a unique $f$ in $D(P)$ solution if and only if
$\int_{V}M(g)d\mu(v)=0$. Moreover it holds that
\begin{equation}
\left\Vert f\right\Vert _{L^{2}(V;B_{A}^{2}(\mathbb{R}_{y}^{d}))}\leq
C\left\Vert g\right\Vert _{L^{2}(V;B_{A}^{2}(\mathbb{R}_{y}^{d}))}
\label{bornfg}%
\end{equation}
where $C>0$ is a constant independent of $g$.

\item[(iii)] Let $\varphi\in L^{2}(V;B_{A}^{2}(\mathbb{R}_{y}^{d}))$. The
equation \emph{(\ref{eq7})} admits a solution $\phi\in D(P^{\ast})$ if and
only if $\int_{V}M(\varphi F)d\mu(v)=0$. The condition $\int_{V}M(\phi
)d\mu(v)=0$ ensures the uniqueness of the solution. Furthermore we have
\begin{equation}
\left\Vert \phi\right\Vert _{L^{2}(V;B_{A}^{2}(\mathbb{R}_{y}^{d}))}\leq
C\left\Vert \varphi\right\Vert _{L^{2}(V;B_{A}^{2}(\mathbb{R}_{y}^{d}))}
\label{bornpvp}%
\end{equation}
where $C>0$ is a constant independent of $\varphi$.
\end{itemize}
\end{proposition}

\begin{proof}
The proof is divided into 5 steps. It is based on the characterization of
$\ker P$, after which we conclude with Fredholm alternative argument
\cite[Theorem VI.6, page 92]{brezis}).

\textit{Step 1: Uniform bounds of the operator }$\mathcal{Q}$.

Let us first and foremost check that $\mathcal{Q}$ is a bounded operator on
$L^{2}(V;B_{A}^{2}(\mathbb{R}_{y}^{d}))$, provided that $\sigma\in L^{\infty
}(V\times V;B_{A}^{2,\infty}(\mathbb{R}_{y}^{d}))$. Indeed for $f\in
L^{2}(V;B_{A}^{2}(\mathbb{R}_{y}^{d}))$, we have
\[
\mathcal{Q}(f)(y,v)=\int_{V}\sigma(y,v,w)(f(y,w)-f(y,v))d\mu(w),
\]
so that
\[
\left\Vert \mathcal{Q}(f)\right\Vert _{L^{2}(V;B_{A}^{2}(\mathbb{R}_{y}^{d}%
))}^{2}=\int_{V}M\left(  \left\vert \mathcal{Q}(f)(\cdot,v)\right\vert
^{2}\right)  d\mu(v).
\]
But
\[
\left\vert \mathcal{Q}(f)(\cdot,v)\right\vert ^{2}\leq C\left\Vert
\sigma(\cdot,v,\cdot)\right\Vert _{L^{\infty}(V\times\mathbb{R}_{y}^{d})}%
^{2}\left(  \int_{V}\left\vert f\right\vert ^{2}d\mu(w)+\left\vert
f\right\vert ^{2}\right)
\]
where $C=\sup(\mu(V),\mu(V)^{2})$. Since the mean value commutes with the
integral, we get
\[
\left\Vert \mathcal{Q}(f)\right\Vert _{L^{2}(V;B_{A}^{2}(\mathbb{R}_{y}^{d}%
))}^{2}\leq C\left\Vert \sigma\right\Vert _{L^{\infty}(V\times V\times
\mathbb{R}_{y}^{d})}^{2}(\mu(V)+1)\left\Vert f\right\Vert _{L^{2}(V;B_{A}%
^{2}(\mathbb{R}_{y}^{d}))}^{2},
\]
that is,
\begin{equation}
\left\Vert \mathcal{Q}(f)\right\Vert _{L^{2}(V;B_{A}^{2}(\mathbb{R}_{y}^{d}%
))}\leq C\left\Vert \sigma\right\Vert _{L^{\infty}(V\times V\times
\mathbb{R}_{y}^{d})}\left\Vert f\right\Vert _{L^{2}(V;B_{A}^{2}(\mathbb{R}%
_{y}^{d}))} \label{BoundQf}%
\end{equation}
where $C$ in (\ref{BoundQf}) depends only on $\mu(V)$. This shows that
$\mathcal{Q}$ is bounded in $L^{2}(V;B_{A}^{2}(\mathbb{R}_{y}^{d}))$.\medskip

\textit{Step 2: Perturbation of the operator }$P$\textbf{.}

We rewrite $P$ as a perturbation of the advection operator $\mathcal{A}%
=a(v)\cdot\nabla_{y}+\Sigma$ (where $\Sigma(y,v)=\int_{V}\sigma(y,w,v)d\mu
(w)$), by the integral operator $K$ defined by
\begin{equation}
Kf(y,v)=\int_{V}\sigma(y,v,w)f(y,w)d\mu(w), \label{Kf}%
\end{equation}
that is,
\begin{equation}
P=\mathcal{A}-K. \label{TAK}%
\end{equation}
$\mathcal{A}$ is invertible; since by a simple integration along the
characteristic lines $y+sa(v)$, we can easily check that the equation
$\mathcal{A}f=h$ has the unique solution
\begin{equation}
f(y,v)=\mathcal{A}^{-1}h(y,v)=\int_{0}^{\infty}\exp\left(  -\int_{0}^{s}%
\Sigma(y-a(v)\tau,v)d\tau\right)  h(y-a(v)s,v)ds, \label{fAh}%
\end{equation}
and the above function $f$ actually lies in $L^{2}(V;B_{A}^{2}(\mathbb{R}%
_{y}^{d}))$. Indeed, thanks to hypothesis (\ref{H2}), we can find
$0<\Sigma_{1}=\inf_{(x,y,v)\in\mathbb{R}_{x}^{d}\times\mathbb{R}_{y}^{d}\times
V}\Sigma(x,y,v)$ (recall that $\Sigma(x,y,v)>0$ since $\sigma>0$), so that for
$h\in L^{2}(V;B_{A}^{2}(\mathbb{R}_{y}^{d}))$, we have
\begin{equation}%
\begin{array}
[c]{l}%
\Vert\mathcal{A}^{-1}h\Vert_{L^{2}(V;B_{A}^{2}(\mathbb{R}_{y}^{d}))}^{2}=\\
=\int_{V}\lim_{R\rightarrow\infty}%
\mathchoice {{\setbox0=\hbox{$\displaystyle{\textstyle
-}{\int}$ } \vcenter{\hbox{$\textstyle -$
}}\kern-.6\wd0}}{{\setbox0=\hbox{$\textstyle{\scriptstyle -}{\int}$ } \vcenter{\hbox{$\scriptstyle -$
}}\kern-.6\wd0}}{{\setbox0=\hbox{$\scriptstyle{\scriptscriptstyle -}{\int}$
} \vcenter{\hbox{$\scriptscriptstyle -$
}}\kern-.6\wd0}}{{\setbox0=\hbox{$\scriptscriptstyle{\scriptscriptstyle
-}{\int}$ } \vcenter{\hbox{$\scriptscriptstyle -$ }}\kern-.6\wd0}}\!\int
_{B_{R}}\left\vert \int_{0}^{\infty}\exp\left(  -\int_{0}^{s}\Sigma
(y-a(v)\tau,v)d\tau\right)  h(y-a(v)s,v)ds\right\vert ^{2}dyd\mu(v)\\
\leq\int_{V}\lim_{R\rightarrow\infty}%
\mathchoice {{\setbox0=\hbox{$\displaystyle{\textstyle
-}{\int}$ } \vcenter{\hbox{$\textstyle -$
}}\kern-.6\wd0}}{{\setbox0=\hbox{$\textstyle{\scriptstyle -}{\int}$ } \vcenter{\hbox{$\scriptstyle -$
}}\kern-.6\wd0}}{{\setbox0=\hbox{$\scriptstyle{\scriptscriptstyle -}{\int}$
} \vcenter{\hbox{$\scriptscriptstyle -$
}}\kern-.6\wd0}}{{\setbox0=\hbox{$\scriptscriptstyle{\scriptscriptstyle
-}{\int}$ } \vcenter{\hbox{$\scriptscriptstyle -$ }}\kern-.6\wd0}}\!\int
_{B_{R}}\left\vert \int_{0}^{\infty}\exp\left(  -\Sigma_{1}s\right)
h(y-a(v)s,v)ds\right\vert ^{2}dyd\mu(v)\\
=\int_{V}\lim_{R\rightarrow\infty}%
\mathchoice {{\setbox0=\hbox{$\displaystyle{\textstyle
-}{\int}$ } \vcenter{\hbox{$\textstyle -$
}}\kern-.6\wd0}}{{\setbox0=\hbox{$\textstyle{\scriptstyle -}{\int}$ } \vcenter{\hbox{$\scriptstyle -$
}}\kern-.6\wd0}}{{\setbox0=\hbox{$\scriptstyle{\scriptscriptstyle -}{\int}$
} \vcenter{\hbox{$\scriptscriptstyle -$
}}\kern-.6\wd0}}{{\setbox0=\hbox{$\scriptscriptstyle{\scriptscriptstyle
-}{\int}$ } \vcenter{\hbox{$\scriptscriptstyle -$ }}\kern-.6\wd0}}\!\int
_{B_{R}}\left\vert \int_{0}^{\infty}\exp\left(  -\frac{\Sigma_{1}}{2}s\right)
\exp\left(  -\frac{\Sigma_{1}}{2}s\right)  h(y-a(v)s,v)ds\right\vert
^{2}dyd\mu(v)\\
\leq\int_{V}\lim_{R\rightarrow\infty}%
\mathchoice {{\setbox0=\hbox{$\displaystyle{\textstyle
-}{\int}$ } \vcenter{\hbox{$\textstyle -$
}}\kern-.6\wd0}}{{\setbox0=\hbox{$\textstyle{\scriptstyle -}{\int}$ } \vcenter{\hbox{$\scriptstyle -$
}}\kern-.6\wd0}}{{\setbox0=\hbox{$\scriptstyle{\scriptscriptstyle -}{\int}$
} \vcenter{\hbox{$\scriptscriptstyle -$
}}\kern-.6\wd0}}{{\setbox0=\hbox{$\scriptscriptstyle{\scriptscriptstyle
-}{\int}$ } \vcenter{\hbox{$\scriptscriptstyle -$ }}\kern-.6\wd0}}\!\int
_{B_{R}}\left(  \int_{0}^{\infty}\exp\left(  -\Sigma_{1}s\right)  ds\right)
\left(  \int_{0}^{\infty}\exp\left(  -\Sigma_{1}s\right)  |h(y-a(v)s,v)|^{2}%
ds\right)  dyd\mu(v)\\
=\left(  \int_{0}^{\infty}\exp\left(  -\Sigma_{1}s\right)  ds\right)  \int
_{V}\int_{0}^{\infty}\exp\left(  -\Sigma_{1}s\right)  \left(  \lim
_{R\rightarrow\infty}\mathchoice {{\setbox0=\hbox{$\displaystyle{\textstyle
-}{\int}$ } \vcenter{\hbox{$\textstyle -$
}}\kern-.6\wd0}}{{\setbox0=\hbox{$\textstyle{\scriptstyle -}{\int}$ } \vcenter{\hbox{$\scriptstyle -$
}}\kern-.6\wd0}}{{\setbox0=\hbox{$\scriptstyle{\scriptscriptstyle -}{\int}$
} \vcenter{\hbox{$\scriptscriptstyle -$
}}\kern-.6\wd0}}{{\setbox0=\hbox{$\scriptscriptstyle{\scriptscriptstyle
-}{\int}$ } \vcenter{\hbox{$\scriptscriptstyle -$ }}\kern-.6\wd0}}\!\int
_{B_{R}}|h(y-a(v)s,v)|^{2}dy\right)  dsd\mu(v)\\
=\frac{1}{\Sigma_{1}}\int_{V}\int_{0}^{\infty}\exp\left(  -\Sigma_{1}s\right)
M\left(  |\tau_{a(v)s}h(\cdot,v)|^{2}\right)  d\mu(v)ds\\
=\frac{1}{\Sigma_{1}}\int_{V}\int_{0}^{\infty}\exp\left(  -\Sigma_{1}s\right)
M\left(  |h(\cdot,v)|^{2}\right)  d\mu(v)ds\\
=\frac{1}{\Sigma_{1}}\int_{V}\int_{0}^{\infty}\exp\left(  -\Sigma_{1}s\right)
M\left(  |h(\cdot,v)|^{2}\right)  d\mu(v)ds\\
=\left(  \frac{1}{\Sigma_{1}}\right)  ^{2}\left\Vert h\right\Vert
_{L^{2}(V;B_{A}^{2}(\mathbb{R}_{y}^{d}))}^{2}=\left(  \frac{1}{\Sigma_{1}%
}\right)  ^{2}\left\Vert h\right\Vert _{L^{2}(V;B_{A}^{2}(\mathbb{R}_{y}%
^{d}))}^{2}.
\end{array}
\label{bornAm1}%
\end{equation}

On the other hand, since $h\geq0$ (resp. $h>0$) implies $f=\mathcal{A}%
^{-1}h\geq0$ (resp. $f=\mathcal{A}^{-1}h>0$), we deduce that $\mathcal{A}%
^{-1}$ is a non negative bounded linear operator on $L^{2}(V;B_{A}%
^{2}(\mathbb{R}_{y}^{d}))$. Now, we rewrite the equation $Pf=g$ as follows:
\begin{equation}
(I-K\circ\mathcal{A}^{-1})h=g,\ h=\mathcal{A}f \label{rewpb}%
\end{equation}
where $I$ is the identity operator in $L^{2}(V;B_{A}^{2}(\mathbb{R}_{y}^{d}%
))$.\medskip

\textit{Step 3: Compactness of }$\mathcal{O}=K\circ\mathcal{A}^{-1}$\textbf{.}

We need to show that under the hypotheses (\ref{H1}) and (\ref{H2}), the
operator $\mathcal{O}=K\circ\mathcal{A}^{-1}$ is compact on $L^{2}(V;B_{A}%
^{2}(\mathbb{R}_{y}^{d}))$. To that end, let us first notice that in view of
(\ref{bornAm1}), $\mathcal{A}^{-1}$ is a bounded linear operator defined from
$L^{2}(V;B_{A}^{2}(\mathbb{R}_{y}^{d}))$ onto itself. It is also a fact that
the operator $K$ sends continuously $L^{2}(V;B_{A}^{2}(\mathbb{R}_{y}^{d}))$
into itself. Thus, following the idea of \cite[Theorem VI.12 (c)]{REEDSIM}, it
suffices to show that $K$ or $\mathcal{A}^{-1}$ is compact. So, let us show
that $K$ is a compact linear operator. To proceed with, let $(f_{n}%
)_{n\in\mathbb{N}}$ be a bounded sequence in $L^{2}(V;B_{A}^{2}(\mathbb{R}%
_{y}^{d}))$. Without losing the generality, we suppose that $(f_{n}%
)_{n\in\mathbb{N}}$ is weakly convergent to some $f\in L^{2}(V;B_{A}%
^{2}(\mathbb{R}_{y}^{d}))$.

We now show that $(Kf_{n})_{n\in\mathbb{N}}$ strongly converges to $Kf\in
L^{2}(V;B_{A}^{2}(\mathbb{R}_{y}^{d}))$ $i.e.$
\[
\lim_{n\rightarrow\infty}\left\Vert Kf_{n}-Kf\right\Vert _{L^{2}(V;B_{A}%
^{2}(\mathbb{R}_{y}^{d}))}^{2}=0.
\]
Under assumption (\ref{H2}) we have, for a.e. $(y,v)\in\mathbb{R}_{y}%
^{d}\times V$,
\[
\int_{V}\left\vert \sigma(y,v,w)\right\vert ^{2}d\mu(w)\leq\mu(V)\left\Vert
\sigma(y,v,\cdot)\right\Vert _{L^{\infty}(V)}^{2}.
\]
Hence for such $(y,v)$, one has
\begin{align}
\lim_{n\rightarrow\infty}Kf_{n}  &  =\lim_{n\rightarrow\infty}\int_{V}%
\sigma(y,v,w)f_{n}(y,w)d\mu(w)\nonumber\\
&  =\lim_{n\rightarrow\infty}\left\langle \sigma(y,v,\cdot),f_{n}%
(y,\cdot)\right\rangle \nonumber\\
&  =\left\langle \sigma(y,v,\cdot),f(y,\cdot)\right\rangle \nonumber\\
&  =\int_{V}\sigma(y,v,w)f(y,w)d\mu(w)\nonumber\\
&  =Kf \label{limitKf}%
\end{align}
where $\left\langle \cdot,\cdot\right\rangle $ stands for duality pairings
between $L^{p^{\prime}}(V)$ and $L^{p}(V)$. Moreover, by Holder inequality, we
have for a.e. $(y,v)\in\mathbb{R}_{y}^{d}\times V$,
\begin{align}
\left\vert Kf(y,v)\right\vert  &  \leq\int_{V}\left\vert \sigma(y,v,w)f_{n}%
(y,w)\right\vert d\mu(w)\nonumber\\
&  \leq\left\Vert f_{n}(y,\cdot)\right\Vert _{L^{2}(V)}\left(  \int
_{V}\left\vert \sigma(y,v,w)\right\vert ^{2}d\mu(w)\right)  ^{\frac{1}{2}%
}\nonumber\\
&  \leq\sup_{n}\left\Vert f_{n}(y,\cdot)\right\Vert _{L^{2}(V)}\left(
\int_{V}\left\vert \sigma(y,v,w)\right\vert ^{p^{\prime}}d\mu(w)\right)
^{\frac{1}{2}}\nonumber\\
&  :=G(y,v). \label{Gyv}%
\end{align}
Thus we have shown in (\ref{limitKf}) that $(Kf_{n}(\cdot,v))_{n}$ converges
pointwise to $Kf(\cdot,v)$ for almost every $v$. On the other hand,
$(Kf_{n}(\cdot,v))_{n}$ is bounded uniformly in $n$ by the function $G(y,v)$,
which is of power $2$ integrable with respect to $d\mu(v)$ for a.e.
$y\in\mathbb{R}_{y}^{d}$ . We infer by the Lebesgue dominated convergence
theorem that for a.e. $y\in\mathbb{R}_{y}^{d}$,
\[
\lim_{n\rightarrow\infty}\left\Vert Kf(y,\cdot)-Kf_{n}(y,\cdot)\right\Vert
_{L^{2}(V)}^{2}=\lim_{n\rightarrow\infty}\int_{V}\left\vert Kf(y,v)-Kf_{n}%
(y,v)\right\vert ^{2}d\mu(v)=0,
\]
which implies that
\begin{align*}
\lim_{n\rightarrow\infty}\left\Vert Kf-Kf_{n}\right\Vert _{L^{2}(V;B_{A}%
^{2}(\mathbb{R}_{y}^{d}))}^{2}\  &  =\lim_{n\rightarrow\infty}\int_{V}M\left(
\left\vert Kf(y,v)-Kf_{n}(y,v)\right\vert ^{2}\right)  d\mu(v)\\
&  =M\left(  \lim_{n\rightarrow\infty}\int_{V}\left\vert Kf(y,v)-Kf_{n}%
(y,v)\right\vert ^{2}d\mu(v)\right) \\
&  =M\left(  0\right)  =0.
\end{align*}
It follows that $K$ is compact in $L^{2}(V;B_{A}^{2}(\mathbb{R}_{y}^{d}))$.
Whence $\mathcal{O}=K\circ\mathcal{A}^{-1}$ is compact in $L^{2}(V;B_{A}%
^{2}(\mathbb{R}_{y}^{d}))$.\medskip

\textit{Step 4: Dimension of }$\ker P$\textit{ by Krein-Rutman theorem.}

Keeping this compactness result in mind and coming back to problem
(\ref{rewpb}), let us consider the following convex closed cone $C=\left\{
f\in L^{2}(V;B_{A}^{2}(\mathbb{R}_{y}^{d})):f\geq0\right\}  $. $C$ is a total
cone because $C-C=L^{2}(V;B_{A}^{2}(\mathbb{R}_{y}^{d}))$. Its interior
$\mathrm{int}C$ is trivially not empty and we have $C\cap(-C)=\left\{
0\right\}  $. Otherwise, since $\sigma>0$ implies $K(f)>0$ for any nonnegative
function $f\in L^{2}(V;B_{A}^{2}(\mathbb{R}_{y}^{d}))$, we deduce that the
compact operator $\mathcal{O}=K\circ\mathcal{A}^{-1}$ sends $C\backslash\{0\}$
into $\mathrm{int}C$. Thus, the positive compact operator $\mathcal{O}%
=K\circ\mathcal{A}^{-1}$ and the convex closed cone $C$ satisfy the conditions
of Theorem \cite[(10.5), P. 281]{Schaefer}, which is the generalization in the
locally convex spaces of the well known Krein-Rutman theorem \cite[Theorem
VI.13, p. 100]{brezis}. Hence, it guaranties the existence of a unique
eigenvalue $\lambda$ of $\mathcal{O}$ (which is its spectral radius)
associated to a non negative eigenfunction $H\geq0$ ($\mathcal{O}(H)=\lambda
H$). We then set $AF=H$, so that $K(F)=\lambda AF$. Applying to this equality,
the mean value $M$ with respect to $y$ and next integrating with respect to
$v$ using the semi-detailed balance condition (\ref{sdbc}) and equality
(\ref{tdnc}), we are lead to
\[
\int_{V\times V}M\left(  \sigma(\cdot,v,w)F(\cdot,w)\right)  d\mu
(w)d\mu(v)=\lambda\int_{V\times V}M\left(  \sigma(\cdot,v,w)F(\cdot,w)\right)
d\mu(w)d\mu(v)>0.
\]
We deduce that $\lambda=1$ is the principal eigenvalue of $\mathcal{O}$.
Furthermore, we have seen in \textit{Step 2} that $\mathcal{A}^{-1}$ and $K$
are non negative linear operators. Thus, since $F=\mathcal{A}^{-1}H\geq0$,
this implies that $AF=H=K(F)>0$, and so $F=\mathcal{A}^{-1}H>0$. Finally, as
$f\geq0$ implies $\mathcal{O}(f)>0$, we deduce that the dimension of the
eigenspace is one. We can proceed in the same way with the adjoint operator
$P^{\ast}$, noticing that $\chi_{\mathbb{R}_{y}^{d}\times V}\in\ker P^{\ast}$,
since constants belong to $\ker P^{\ast}$. To conclude the proof of (i), let
us show that (\ref{eq17}) is satisfied.

Let $\varphi\in\mathcal{C}_{0}^{\infty}(\mathbb{R}^{d})$ and $\phi
\in\mathcal{C}^{\infty}(V;A^{\infty})$ be freely fixed, and define the
function $\psi(y,v)=\varphi(\varepsilon y)\phi(y,v)$ ($y\in\mathbb{R}^{d}$)
for a given $\varepsilon>0$. We multiply the equation $PF=0$ by $\psi$ and
integrate by parts the resulting equation over $\mathbb{R}^{d}\times V$, and
next make the change of variables $z=\varepsilon y$ to get
\begin{equation}
\iint_{\mathbb{R}^{d}\times V}F^{\varepsilon}\left[  a(v)\cdot(\varepsilon
\phi^{\varepsilon}\nabla_{z}\varphi+\varphi(\nabla_{y}\phi)^{\varepsilon
})+\varphi(\mathcal{Q}^{\ast}\phi)^{\varepsilon}\right]  dzd\mu(v)=0
\label{eq20}%
\end{equation}
where $\eta^{\varepsilon}(z,v)=\eta(z/\varepsilon,v)$. Letting $\varepsilon
\rightarrow0$ in (\ref{eq20}) using the definition of the sigma-convergence
(for time independent sequences) we obtain
\[
\iint_{\mathbb{R}^{d}\times V}M(F(a(v)\cdot\nabla_{y}\phi(\cdot,v)+\mathcal{Q}%
^{\ast}\phi(\cdot,v)))\varphi(z)dzd\mu(v)=0.
\]
The arbitrariness of $\varphi$ yields at once
\[
\int_{V}M(F(a(v)\cdot\nabla_{y}\phi(\cdot,v)+\mathcal{Q}^{\ast}\phi
(\cdot,v)))d\mu(v)=0,
\]
that is (\ref{eq17}). This finishes the proof of (i).\medskip

\textit{Step 5: The Fredholm Alternative.}

In order to apply the Fredholm Alternative, we need to consider the extension
$\mathcal{P}$ of $P$ defined on $L^{2}(V;\mathcal{B}_{A}^{2}(\mathbb{R}%
_{y}^{d}))$. Then all the properties obtained in Steps 1-4 are valid mutatis
mutandis (replace $P$ by $\mathcal{P}$ and $L^{2}(V;B_{A}^{2}(\mathbb{R}%
_{y}^{d}))$ by $L^{2}(V;\mathcal{B}_{A}^{2}(\mathbb{R}_{y}^{d}))$). We also
observe that (\ref{eq6}) may be seen as equivalent to the following one
\begin{equation}
\mathcal{P}\widetilde{f}=\widetilde{g}%
.\ \ \ \ \ \ \ \ \ \ \ \ \ \ \ \ \ \ \ \ \ \ \ \ \ \ \ \ \ \ \ \ \ \ \ \ \ \ \ \ \ \label{eq9}%
\end{equation}
Indeed (\ref{eq6}) implies (\ref{eq9}) while the reverse implication holds
provided that we replace $g$ by $g+\psi$ where $\psi\in L^{2}(V;\mathcal{N})$.
However in the variational form of (\ref{eq6}) we observe that the
contribution of $\psi$ is of no effect since $M(\psi(\cdot,v)\phi(\cdot,v))=0$
for $\psi\in L^{2}(V;\mathcal{N})$ and any $\phi\in L^{2}(V;B_{A}%
^{2}(\mathbb{R}_{y}^{d}))$.

This being so, we note that the condition $\int_{V}M(g)d\mu(v)=0$ is necessary
for the solvability of $Pf=g$, since $\int_{V}M(Pf)d\mu(v)=0$. Indeed, as the
mean value mapping $u\mapsto M(u)$ commutes with integral with respect to $v$,
we simply notice that combining (\ref{Lf1}) and (\ref{tdnc}), we have
\[
\int_{V}\mathcal{Q}f(y,v)d\mu(v)=0\text{ for all }f\in L^{2}(V;B_{A}%
^{2}(\mathbb{R}_{y}^{d})).
\]
Furthermore, by a simple computation, we show that $M\left(  a(v)\cdot
\nabla_{y}f(\cdot,v)\right)  =0$.

With this in mind, we may replace (\ref{eq6}) by (\ref{eq9}) and we remark
that (\ref{eq9}) is equivalent to (\ref{rewpb}) in which we consider the
corresponding extensions of the operators involved in (\ref{rewpb}) and
defined on $L^{2}(V;\mathcal{B}_{A}^{2}(\mathbb{R}_{y}^{d}))$, so that
applying the Fredholm alternative, we have $\operatorname{Im}(I-\mathcal{O}%
)=\ker(I-\mathcal{O}^{\ast})^{\bot}$, which implies $\operatorname{Im}P=(\ker
P^{\ast})^{\bot}$. Hence, $Pf=g$ is solvable for $g\in\operatorname{Im}P=(\ker
P^{\ast})^{\bot}$. But as the eigenspaces of $P$ and $P^{\ast}$ are spanned by
positive functions, the condition $\int_{V}M(g)d\mu(v)=0$ guarantees the
uniqueness. So, for any element in $L_{0}^{2}(V;B_{A}^{2}(\mathbb{R}_{y}%
^{d}))=\{g\in L^{2}(V;B_{A}^{2}(\mathbb{R}_{y}^{d})):\int_{V}M(g(\cdot
,w))d\mu(w)=0\}$, we find a unique $f\in L_{0}^{2}(V;B_{A}^{2}(\mathbb{R}%
_{y}^{d}))$ solution of $Pf=g$. To prove inequality (\ref{bornfg}), we just
apply the open mapping theorem (see, for instance \cite[p. 18]{brezis}) to the
linear operator $\mathcal{P}$, which gives the existence of $C>0$ such that
\[
\left\Vert f\right\Vert _{L^{2}(V;B_{A}^{2}(\mathbb{R}_{y}^{d}))}\leq
C\left\Vert g\right\Vert _{L^{2}(V;B_{A}^{2}(\mathbb{R}_{y}^{d}))}.
\]
Similar conclusions hold for the adjoint problem $P^{\ast}\varphi=\phi$. This
concludes the proof.
\end{proof}

The next two results show how the regularity of the coefficients gives \ rise
to the regularity of the solutions of (\ref{eq6}) and (\ref{eq7}). The second
one takes into account the dependence with respect to the parameter $x$. Their
proofs are copied on that of \cite[Lemmas 3.2 and 3.3]{Goudon-Mellet}.

In what follows, $A^{1}$ denotes the space $\{u\in A:\nabla u\in(A)^{d}\}$
while $V_{\omega}$ stands for the space $V$ with generic variable $\omega\in
V$.

\begin{lemma}
\label{lemmeregu1}Suppose that $\sigma(y,v,w)$ and $\frac{\partial\sigma
}{\partial y_{i}}(y,v,w)$ ($1\leq i\leq d$) lie in $\mathcal{C}(V_{v}\times
V_{w};L^{\infty}(\mathbb{R}_{y}^{d}))\cap L^{\infty}(V_{v}\times V_{w}%
\times\mathbb{R}_{y}^{d})$. Then, for $g\in\mathcal{C}(V;A^{1})$, the solution
of \emph{(\ref{eq6})} lies in $\mathcal{C}(V;A^{1})$. A similar conclusion
holds for the adjoint equation \emph{(\ref{eq7})}.
\end{lemma}

\begin{proof}
Let $\sigma$, $\frac{\partial\sigma}{\partial y_{i}}$ and $g$ be given as in
the statement of the lemma. In particular, we have $g\in\mathcal{C}%
(V;A^{1})\subset\mathcal{C}(V;A)\subset\mathcal{C}(V;B_{A}^{2}(\mathbb{R}%
_{y}^{d}))$. Proceeding as in the previous proof (see \emph{Step 4}), we show
that the solution $h$ of the following rewritten problem
\begin{equation}
(I-K\circ\mathcal{A}^{-1})(h)=g,\ h=\mathcal{A}(f) \label{rewpb2}%
\end{equation}
belongs to $\mathcal{C}(V;A)$. To conclude, we need to show that $h$ and
$f=\mathcal{A}^{-1}h$ have the same regularity. To that end, Thanks to
(\ref{fAh}), we have
\[
f(y,v)=\int_{0}^{\infty}\exp\left(  -\int_{0}^{s}\Sigma(y-a(v)\tau
,v)d\tau\right)  h(y-a(v)s,v)ds,
\]
which implies that $f\in\mathcal{C}(V;A)$.

Now, let us fix $1\leq i\leq d$, and set $\partial_{i}=\partial/\partial
y_{i}$. Applying $\partial_{i}$ to (\ref{eq6}) yields
\begin{equation}
a(v)\cdot\nabla_{y}\partial_{i}f-\mathcal{Q}(\partial_{i}f)=\partial
_{i}g+\partial_{i}\mathcal{Q}(f), \label{opderive}%
\end{equation}
where
\begin{equation}
\partial_{i}\mathcal{Q}(f)(y,v)=\int_{V}\partial_{i}\sigma(y,v,w)f(y,w)d\mu
(w)-\partial_{i}\Sigma(y,v)f(y,v). \label{deriveQ}%
\end{equation}
Since $f\in\mathcal{C}(V;A)$, we have $\partial_{i}g$, $\partial
_{i}\mathcal{Q}(f)\in\mathcal{C}(V;A)$. Therefore, the right-hand side of
(\ref{opderive}) belongs to $\mathcal{C}(V;A)$. Thanks to the Proposition
\ref{propkers} applying this time to problem (\ref{opderive}), we obtain
$\partial_{i}f\in\mathcal{C}(V;A)$. Hence $f\in\mathcal{C}(V;A^{1})$. The same
arguments can be applied to the adjoint problem.
\end{proof}

It is worth noticing that derivation of the equation for the equilibrium
function $F$ shows similarly that $F\in\mathcal{C}(V;A^{1})$.

\begin{lemma}
\label{lemmeregu2}Let $k\in\mathbb{N}^{\ast}$. If $\sigma(x,y,v,w)\in
\mathcal{C}^{k}(\mathbb{R}_{x}^{d};L^{\infty}(V\times V;B_{A}^{2,\infty
}(\mathbb{R}_{y}^{d})))$ and $g\in\mathcal{C}^{k}(\mathbb{R}_{x}^{d};L^{2}(V;$
$B_{A}^{2}(\mathbb{R}_{y}^{d})))$, then the solution of \emph{(\ref{eq6})}
lies in $\mathcal{C}^{k}(\mathbb{R}_{x}^{d};L^{2}(V;B_{A}^{2}(\mathbb{R}%
_{y}^{d})))$. A similar conclusion holds for the adjoint equation
\emph{(\ref{eq7})}.
\end{lemma}

\begin{proof}
We first check that the constant $C$ in Proposition \ref{propkers}, which
depends on the parameter $x$, is locally bounded. Indeed, let $K$ be a compact
subset of $\mathbb{R}^{d}$ and $0<\delta<1$. Applying Heine's theorem to the
restriction of $x\mapsto\sigma(x,\cdot)$ to $K$, we derive the existence of a
constant $\eta_{\delta}>0$ such that, for any $x_{1},x_{2}\in K$ with
$\left\vert x_{1}-x_{2}\right\vert \leq\eta_{\delta}$, we have $\left\Vert
\sigma(x_{1},\cdot)-\sigma(x_{2},\cdot)\right\Vert _{L^{\infty}(Y\times
V\times V)}\leq\frac{\delta}{C(x_{2})C}$, where $C$ is the constant obtained
in (\ref{BoundQf}), depending only on $\mu(V)<\infty$ and $C(x_{2})$ the
constant in Proposition \ref{propkers}. Accordingly
\[
\left\Vert \mathcal{Q}(x)-\mathcal{Q}(x^{\prime})\right\Vert _{\mathcal{L}%
\left(  L^{2}(V;B_{A}^{2}(\mathbb{R}_{y}^{d}))\right)  }\leq\frac{\delta
}{C(x^{\prime})}
\]
where $\mathcal{L}\left(  L^{2}(V;B_{A}^{2}(\mathbb{R}_{y}^{d}))\right)  $
stands for the space of bounded linear operators from $L^{2}(V;B_{A}%
^{2}(\mathbb{R}_{y}^{d}))$ into itself. Let $(B(x_{i},\eta_{\delta}))_{1\leq
i\leq I_{\delta}}$ be a finite covering of $K$: $K\subset\cup_{i=1}%
^{I_{\delta}}B(x_{i},\eta_{\delta})$. For any $x\in K$, we rewrite the
equation $P(x)(f(x))=g(x)$ under the form
\[
P(x_{i})(f(x))=g(x)+(\mathcal{Q}(x)-\mathcal{Q}(x_{i}))(f(x)).
\]
It follows that
\begin{align*}
\left\Vert P(x_{i})f(x)\right\Vert _{L^{2}(V;B_{A}^{2}(\mathbb{R}_{y}^{d}))}
&  \leq\left\Vert g(x)\right\Vert _{L^{2}(V;B_{A}^{2}(\mathbb{R}_{y}^{d}%
))}+\left\Vert \mathcal{Q}(x)-\mathcal{Q}(x_{i})\right\Vert _{\mathcal{L}%
\left(  L^{2}(V;B_{A}^{2}(\mathbb{R}_{y}^{d}))\right)  }\left\Vert
f(x)\right\Vert _{L^{2}(V;B_{A}^{2}(\mathbb{R}_{y}^{d}))}\\
&  \leq\left\Vert g(x)\right\Vert _{L^{2}(V;B_{A}^{2}(\mathbb{R}_{y}^{d}%
))}+\frac{\delta}{C(x_{i})}\left\Vert f(x)\right\Vert _{L^{2}(V;B_{A}%
^{2}(\mathbb{R}_{y}^{d}))}.
\end{align*}
But thanks to (\ref{bornfg}), we have
\begin{align*}
\frac{1}{C(x_{i})}\left\Vert f(x)\right\Vert _{L^{2}(V;B_{A}^{2}%
(\mathbb{R}_{y}^{d}))}  &  \leq\left\Vert g(x)\right\Vert _{L^{2}(V;B_{A}%
^{2}(\mathbb{R}_{y}^{d}))}\\
&  =\left\Vert P(x)f(x)\right\Vert _{L^{2}(V;B_{A}^{2}(\mathbb{R}_{y}^{d}))}\\
&  \leq\left\Vert g(x)\right\Vert _{L^{2}(V;B_{A}^{2}(\mathbb{R}_{y}^{d}%
))}+\frac{\delta}{C(x_{i})}\left\Vert f(x)\right\Vert _{L^{2}(V;B_{A}%
^{2}(\mathbb{R}_{y}^{d}))}.
\end{align*}
We then deduce that
\begin{equation}
\left\Vert f(x)\right\Vert _{L^{2}(V;B_{A}^{2}(\mathbb{R}_{y}^{d}))}\leq
\frac{C(x_{i})}{1-\delta}\left\Vert g(x)\right\Vert _{L^{2}(V;B_{A}%
^{2}(\mathbb{R}_{y}^{d}))}. \label{bornf(x)}%
\end{equation}
Hence, we can use the constant $C=C(K,\delta)=\max\left\{  \frac{C(x_{i}%
)}{1-\delta};i=1,...,I_{\delta}\right\}  $.

Now, let $\delta_{h}^{i}$ denotes a differential quotient with respect to $x$:
$\delta_{h}^{i}f(x,\cdot)=\frac{f(x+he_{i},\cdot)-f(x,\cdot)}{h}$ for a fixed
$1\leq i\leq d$. We get from $Pf=g$ the following equation
\begin{equation}
a(v)\cdot\nabla_{y}\delta_{h}^{i}f-Q(\delta_{h}^{i}f)=\delta_{h}^{i}%
g+\delta_{h}^{i}Q(f). \label{eq8}%
\end{equation}
Arguing as in the proof of Lemma \ref{lemmeregu1}, we show that the right-hand
side of (\ref{eq8}) is bounded in $L^{2}(V;B_{A}^{2}(\mathbb{R}_{y}^{d}))$ and
so, $\left(  \delta_{h}^{i}f\right)  _{h>0}$ is also bounded in $L^{2}%
(V;B_{A}^{2}(\mathbb{R}_{y}^{d}))$. Therefore, the application $x\mapsto
f(x,\cdot)$ is differentiable with values in $L^{2}(V;B_{A}^{2}(\mathbb{R}%
_{y}^{d}))$, and we have, setting $\partial_{i}=\partial/\partial x_{i}$
\[
a(v)\cdot\nabla_{y}\partial_{i}f-Q(\partial_{i}f)=\partial_{i}g+\partial
_{i}Q(f).
\]
We easily show the continuity of the right-hand side with respect to $x$ using
the continuity of $f$, the hypothesis on $\sigma$ and (\ref{deriveQ}). From
that, we have the continuity of $\partial_{i}f$. These arguments can repeat
for higher derivatives and apply to the adjoint problem.
\end{proof}

In particular, the following consequence of the lemmas above will be useful in
the sequel.

\begin{corollary}
\label{corolregF}Suppose \emph{(\ref{H1})} and \emph{(\ref{H2})} hold. Then

\begin{itemize}
\item[(i)] $F$ and $\nabla_{x}F$ are continuous functions of their arguments,
where $F$ is determined by \emph{(\ref{KerP})}.

\item[(ii)] For any $g\in\mathcal{C}^{1}(\mathbb{R}_{x}^{d};\mathcal{C}%
(V;A^{1}))$, satisfying the compatibility condition in Proposition
\emph{\ref{propkers}}, the solution of $Pf=g$ is a continuous function of its
arguments as well as its first derivative with respect to $x$. A similar
conclusion holds for the adjoint equation $P^{\ast}\varphi=\phi$.
\end{itemize}
\end{corollary}

\section{Homogenization result: proof of Theorem \ref{theosol}\label{sec5}}

Let $F$ be the unique solution of (\ref{KerP}) in Proposition \ref{propkers}.
We assume here that (\ref{H2}) and (\ref{vfc}) hold true.

\begin{proof}
[Proof of Theorem \emph{\ref{theosol}}]Let $E=(\varepsilon_{n})_{n}$ denotes
any ordinary sequence of positive real numbers satisfying $\varepsilon
_{n}\rightarrow0$ when $n\rightarrow\infty$. We denote the generic element of
$E$ by $\varepsilon$ such that $"\varepsilon\rightarrow0"$ amounts to
$"\varepsilon_{n}\rightarrow0$ when $n\rightarrow\infty"$. By virtue of a
priori estimates (\ref{estimf}), we appeal to Theorem \ref{Theohom1} derive
the existence of a subsequence $E^{\prime}$ of $E$ and element $f_{0}\in
L^{2}(\mathbb{R}_{T}^{d}\times V;B_{A}^{2}(\mathbb{R}_{y}^{d}))$, such that as
$E^{\prime}\ni\varepsilon\rightarrow0$, one has
\begin{equation}
f_{\varepsilon}\rightarrow f_{0}+\mathcal{N}\text{ in }L^{2}(\mathbb{R}%
_{T}^{d}\times V)\text{-weak }\Sigma. \label{converg}%
\end{equation}

Thanks to assumption (\ref{H2}) on the absorption rate $\sigma$, it belongs to
$\mathcal{B}(\mathbb{R}^{d}\times V\times V;B_{A}^{2,\infty}(\mathbb{R}%
_{y}^{d}))$. Now let $\psi\in\mathcal{C}_{0}^{\infty}(\mathbb{R}_{T}^{d}\times
V)\otimes A^{\infty}$ and define $\psi^{\varepsilon}$ by $\psi^{\varepsilon
}(t,x,v)=\psi(t,x,\frac{x}{\varepsilon},v)$ ($(t,x,v)\in(0,T)\times
\mathbb{R}^{d}\times V$). Then $\psi^{\varepsilon}\in\mathcal{C}_{0}^{\infty
}(\mathbb{R}_{T}^{d}\times V)$ where $\mathbb{R}_{T}^{d}=(0,T)\times
\mathbb{R}_{x}^{d}$. Multiplying the first equation in (\ref{EH0}) by
$\frac{1}{\varepsilon}\psi^{\varepsilon}$, and integrating by parts using the
relation (\ref{intvuLw}) and the fact that $\nabla_{x}\cdot a(v)=0$, we
obtain
\begin{equation}
-\int_{V\times\mathbb{R}_{T}^{d}}\left[  f_{\varepsilon}\left(  \left(
\frac{\partial\psi}{\partial t}\right)  ^{\varepsilon}+\frac{1}{\varepsilon
}a(v)\cdot(\nabla_{x}\psi)^{\varepsilon}+\frac{1}{\varepsilon^{2}}%
a(v)\cdot(\nabla_{y}\psi)^{\varepsilon}+\frac{1}{\varepsilon^{2}}%
(\mathcal{Q}^{\ast}\psi)^{\varepsilon}\right)  \right]  dtdxd\mu(v)=0
\label{eqfaible}%
\end{equation}
where $\mathcal{Q}^{\ast}$ is defined by (\ref{adjOpQ}). Multiplying
(\ref{eqfaible}) by $\varepsilon^{2}$ and passing to the limit (thanks to
Theorem \ref{Theohom1}) as $\varepsilon\rightarrow0$, one has
\begin{align*}
&  \int_{\mathbb{R}_{T}^{d}\times V}\left[  f_{\varepsilon}(\varepsilon
^{2}\left(  \frac{\partial\psi}{\partial t}\right)  ^{\varepsilon}+\varepsilon
a(v)\cdot(\nabla_{x}\psi)^{\varepsilon}+a(v)\cdot(\nabla_{y}\psi
)^{\varepsilon})\right]  dtdxd\mu(v)\\
&  \rightarrow\int_{\mathbb{R}_{T}^{d}\times V}M\left(  f_{0}(a(v)\cdot
\nabla_{y}\psi)\right)  dtdxd\mu(v).
\end{align*}
Next, using $\sigma$ as a test function, we obtain, as $E^{\prime}%
\ni\varepsilon\rightarrow0$,
\begin{align*}
&  \int_{\mathbb{R}_{T}^{d}\times V}f_{\varepsilon}(\mathcal{Q}^{\ast}%
\psi)^{\varepsilon}dtdxd\mu(v)=\\
&  =\int_{\mathbb{R}_{T}^{d}\times V^{2}}[(\sigma^{\varepsilon}%
(x,v,w)f_{\varepsilon}(t,x,v)-\sigma^{\varepsilon}(x,w,v)f_{\varepsilon
}(t,x,w))\psi^{\varepsilon}]d\mu(w)d\mu(v)dtdx\\
&  \rightarrow\int_{\mathbb{R}_{T}^{d}\times V^{2}}M[(\sigma(x,\cdot
,v,w)f_{0}-\sigma(x,\cdot,w,v)f_{0})\psi]dtdxd\mu(w)d\mu(v)\\
&  =\int_{\mathbb{R}_{T}^{d}\times V}M(f_{0}(\mathcal{Q}^{\ast}\psi
))dtdxd\mu(v).
\end{align*}
We deduce that
\begin{equation}
-\int_{\mathbb{R}_{T}^{d}\times V}M(f_{0}(a(v)\cdot\nabla_{y}\psi
+\mathcal{Q}^{\ast}\psi))dtdxd\mu(v)=0. \label{eqfaible1}%
\end{equation}
Choosing $\psi(t,x,y,v)=\varphi(t,x)\phi(y,v)$ where $\varphi\in
\mathcal{C}_{0}^{\infty}(\mathbb{R}_{T}^{d})\text{ and }\phi\in A^{\infty
}\otimes\mathcal{C}^{\infty}(V)$, we obtain that (\ref{eqfaible1}) is
equivalent to
\[
-\int_{\mathbb{R}_{T}^{d}}\left(  \int_{V}M\left(  f_{0}(t,x,\cdot,v)\left(
a(v)\cdot\nabla_{y}\phi(\cdot,v)+\mathcal{Q}^{\ast}\phi(\cdot,v)\right)
\right)  d\mu(v)\right)  \varphi(t,x)dtdx=0.
\]
Since $\varphi$ is arbitrarily chosen, the preceding equation leads to
\begin{equation}
-\int_{V}M\left(  f_{0}(t,x,\cdot,v)\left(  a(v)\cdot\nabla_{y}\phi
(\cdot,v)+\mathcal{Q}^{\ast}\phi(\cdot,v)\right)  \right)  d\mu(v)=0.
\label{eq15}%
\end{equation}
So let us consider, for $(t,x)\in\mathbb{R}_{T}^{d}$ be fixed, the equation:
Find $F(t,x)\equiv F\in L^{2}(V;B_{A}^{2}(\mathbb{R}_{y}^{d}))$ such that
\begin{equation}
a(v)\cdot\nabla_{y}F-\mathcal{Q}F=0\text{ in }\mathbb{R}_{y}^{d}\times
V.\ \ \ \ \ \ \ \ \ \ \ \ \ \ \ \ \label{eq16}%
\end{equation}
Then appealing to [part (i) of] Proposition \ref{propkers}, there exists a
unique $F\in L^{2}(V;B_{A}^{2}(\mathbb{R}_{y}^{d}))$ with $\int_{V}%
M(F)d\mu(v)=1$, which solves (\ref{eq16}). As in the proof of (\ref{eq17}), we
observe that $F$ satisfies the variational formulation
\begin{equation}
\int_{V}M\left(  F(\cdot,v)\left(  a(v)\cdot\nabla_{y}\phi(\cdot
,v)+\mathcal{Q}^{\ast}\phi(\cdot,v)\right)  \right)  d\mu(v)=0\ \forall\phi\in
A^{\infty}\otimes\mathcal{C}^{\infty}(V). \label{eq19}%
\end{equation}
Next, set
\[
\rho_{0}(t,x)=\int_{V}M(f_{0}(t,x,\cdot,v))d\mu
(v)\ \ \ \ \ \ \ \ \ \ \ \ \ \ \ \ \ \ \ \ \ \ \ \
\]
and assume without lost of generality that $\rho_{0}$ is not identically zero.
Then this defines a function with the property that $\rho_{0}^{-1}%
(t,x)f_{0}(t,x,\cdot,\cdot)$ solves (\ref{eq19}) and $\int_{V}M(\rho_{0}%
^{-1}(t,x)f_{0}(t,x,\cdot,v))d\mu=1$. Invoking the uniqueness of the solution
of (\ref{eq19}) with the further normality condition $\int_{V}M(F)d\mu(v)=1$,
we get readily $F=\rho_{0}^{-1}f_{0}$, i.e.
\begin{equation}
f_{0}(t,x,y,v)=\rho_{0}(t,x)F(x,y,v)\text{ for a.e. }(t,x,y,v)\in
\mathbb{R}_{T}^{d}\times\mathbb{R}_{y}^{d}\times V. \label{solution}%
\end{equation}
Moreover, recalling that $f_{0}\in L^{2}(\mathbb{R}_{T}^{d}\times V;B_{A}%
^{2}(\mathbb{R}_{y}^{d}))=L^{2}(\mathbb{R}_{T}^{d};L^{2}(V;B_{A}%
^{2}(\mathbb{R}_{y}^{d})))$ we see that $\rho_{0}\in L^{2}(\mathbb{R}_{T}%
^{d})$.

Now let $\phi\in\mathcal{C}_{0}^{\infty}(\mathbb{R}_{T}^{d})$ and define
$\varphi$ by
\begin{equation}
P^{\ast}\varphi=-a(v)\cdot\nabla_{x}\phi,\ \int_{V}M(\varphi)d\mu
=0.\label{testfunc}%
\end{equation}
We recall that in view of (\ref{vfc}), Proposition \ref{propkers} and Lemma
\ref{lemmeregu2} ensure the existence of a unique $\varphi\in\mathcal{C}%
_{0}^{\infty}(\mathbb{R}_{x}^{d};L^{2}(V;B_{A}^{2}(\mathbb{R}_{y}^{d})))$
satisfying (\ref{testfunc}). Moreover if we consider the vector valued
function $\chi^{\ast}=(\chi_{j}^{\ast})_{1\leq j\leq d}$ defined by
\begin{equation}
\left\{
\begin{array}
[c]{l}%
\chi_{j}^{\ast}\in\mathcal{C}_{0}^{\infty}(\mathbb{R}_{x}^{d};L^{2}%
(V;B_{A}^{2}(\mathbb{R}_{y}^{d})))\\
P^{\ast}(\chi_{j}^{\ast})=-a_{j}(v)\text{ and }\int_{V}M(\chi_{j}^{\ast
}(x,\cdot,v))d\mu(v)=0,
\end{array}
\right.  \label{chi}%
\end{equation}
then from the uniqueness argument, it holds that
\[
\varphi(t,x,y,v)=\chi^{\ast}(x,y,v)\cdot\nabla_{x}\phi(t,x).
\]
This being so, we choose in the variational form of (\ref{EH0}) (see
(\ref{eqfaible})) the test function $\frac{1}{\varepsilon}\psi^{\varepsilon}$
where
\[
\psi^{\varepsilon}(t,x,v)=\phi(t,x)+\varepsilon\varphi\left(  t,x,\frac
{x}{\varepsilon},v\right)
\]
with $\phi$ be given above and $\varphi$ being defined by (\ref{testfunc}).
Then we get
\begin{equation}%
\begin{array}
[c]{l}%
\int_{\mathbb{R}_{T}^{d}\times V}f_{\varepsilon}[\frac{\partial\phi}{\partial
t}+\varepsilon\left(  \frac{\partial\varphi}{\partial t}\right)
^{\varepsilon}+\frac{1}{\varepsilon}a(v)\cdot\nabla_{x}\phi+a(v)\cdot
(\nabla_{x}\varphi)^{\varepsilon}\\
\ \ \ +\frac{1}{\varepsilon^{2}}(P^{\ast}\phi)^{\varepsilon}+\frac
{1}{\varepsilon}(P^{\ast}\varphi)^{\varepsilon}]dtdxd\mu(v)=0.
\end{array}
\label{formfaible0}%
\end{equation}
Since $\phi$ is independent of $y$ and $v$, we have $P^{\ast}\phi=0$. We infer
from (\ref{testfunc}) that $a(v)\cdot\nabla_{x}\phi+P^{\ast}\varphi=0$, so
that (\ref{formfaible0}) becomes
\[
\int_{\mathbb{R}_{T}^{d}\times V}\left[  f_{\varepsilon}\left(  \frac
{\partial\phi}{\partial t}+\varepsilon\left(  \frac{\partial\varphi}{\partial
t}\right)  ^{\varepsilon}+a(v)\cdot(\nabla_{x}(\chi^{\ast}\cdot\nabla_{x}%
\phi))^{\varepsilon}\right)  \right]  dtdxd\mu(v)=0.
\]
Letting $\varepsilon\rightarrow0$ (up to a subsequence),
\begin{equation}
\int_{\mathbb{R}_{T}^{d}\times V}M\left(  f_{0}\left(  \frac{\partial\phi
}{\partial t}+a(v)\cdot\nabla_{x}(\chi^{\ast}\cdot\nabla_{x}\phi)\right)
\right)  dtdxd\mu(v)=0.\label{pbmacro1}%
\end{equation}
We can compute each term of (\ref{pbmacro1}) as follows:
\begin{align}
\iint_{\mathbb{R}_{T}^{d}\times V}M\left(  f_{0}\frac{\partial\phi}{\partial
t}\right)  dtdxd\mu(v) &  =\int_{\mathbb{R}_{T}^{d}}\int_{V}M\left(  \rho
_{0}(x,t)F(x,\cdot,v)\frac{\partial\phi}{\partial t}\right)  dtdxd\mu
(v)\label{pbmacro2}\\
&  =\int_{\mathbb{R}_{T}^{d}}\rho_{0}(x,t)\left(  \int_{V}M\left(
F(x,\cdot,v)\right)  d\mu(v)\right)  \frac{\partial\phi}{\partial
t}dtdx\nonumber\\
&  =\int_{\mathbb{R}_{T}^{d}}\rho_{0}(x,t)\frac{\partial\phi}{\partial
t}dtdx\nonumber\\
&  =-\left\langle \frac{\partial\rho_{0}}{\partial t},\phi\right\rangle
\nonumber
\end{align}
since $\int_{V}M(F(x,\cdot,v))d\mu(v)=1$. Next,
\begin{align*}
&  \iint_{\mathbb{R}_{T}^{d}\times V}M\left(  f_{0}a(v)\cdot\nabla_{x}%
(\chi^{\ast}\cdot\nabla_{x}\phi)\right)  dtdxd\mu(v)\\
&  =\iint_{\mathbb{R}_{T}^{d}\times V}M\left(  \sum_{j=1}^{d}Fa_{j}%
(v)\frac{\partial}{\partial x_{j}}(\chi^{\ast}\cdot\nabla_{x}\phi)\right)
\rho_{0}dxdtd\mu\\
&  =\sum_{j=1}^{d}\iint_{\mathbb{R}_{T}^{d}\times V}M\left(  Fa_{j}%
(v)\frac{\partial}{\partial x_{j}}(\chi^{\ast}\cdot\nabla_{x}\phi)\right)
\rho_{0}dxdtd\mu\\
&  =\sum_{j=1}^{d}\iint_{\mathbb{R}_{T}^{d}\times V}M\left[  \frac{\partial
}{\partial x_{j}}(Fa_{j}(v)(\chi^{\ast}\cdot\nabla_{x}\phi))-a_{j}%
(v)(\chi^{\ast}\cdot\nabla_{x}\phi)\frac{\partial F}{\partial x_{j}}\right]
\rho_{0}dxdtd\mu\\
&  =\sum_{j=1}^{d}\iint_{\mathbb{R}_{T}^{d}\times V}\left[  \frac{\partial
}{\partial x_{j}}M\left(  (Fa_{j}(v)(\chi^{\ast}\cdot\nabla_{x}\phi)\right)
-M\left(  a_{j}(v)(\chi^{\ast}\cdot\nabla_{x}\phi)\frac{\partial F}{\partial
x_{j}}\right)  \right]  \rho_{0}dxdtd\mu\\
&  =I_{1}+I_{2}.
\end{align*}
We first deal with $I_{2}$ to obtain
\begin{align*}
I_{2} &  =-\sum_{j=1}^{d}\iint_{\mathbb{R}_{T}^{d}\times V}M\left(
a_{j}(v)(\chi^{\ast}\cdot\nabla_{x}\phi)\frac{\partial F}{\partial x_{j}%
}\right)  \rho_{0}dxdtd\mu\\
&  =-\iint_{\mathbb{R}_{T}^{d}\times V}M\left(  (a(v)\cdot\nabla_{x}%
F)(\chi^{\ast}\cdot\nabla_{x}\phi)\right)  \rho_{0}dxdtd\mu\\
&  =-\sum_{i=1}^{d}\iint_{\mathbb{R}_{T}^{d}\times V}\rho_{0}M\left(
(a(v)\cdot\nabla_{x}F)\chi_{i}^{\ast}\right)  \frac{\partial\phi}{\partial
x_{i}}dxdtd\mu.
\end{align*}
At this level, we set $U(x)=(U_{i}(x))_{1\leq i\leq d}$ where
\begin{equation}
U_{i}(x)=\int_{V}M\left(  \chi_{i}^{\ast}(a(v)\cdot\nabla_{x}F)\right)
d\mu.\label{1.10}%
\end{equation}
Then $U_{i}\in\mathcal{C}_{0}^{\infty}(\mathbb{R}^{d})$ and
\[
I_{2}=\sum_{i=1}^{d}\int_{\mathbb{R}_{T}^{d}}\rho_{0}U_{i}\frac{\partial\phi
}{\partial x_{i}}dxdt=\left\langle \Div_{x}(U\rho_{0}),\phi\right\rangle .
\]
As regard $I_{1}$, we have
\[
I_{1}=\sum_{j=1}^{d}\iint_{\mathbb{R}_{T}^{d}\times V}\rho_{0}\frac{\partial
}{\partial x_{j}}M\left(  (Fa_{j}(v)(\chi^{\ast}\cdot\nabla_{x}\phi)\right)
dxdtd\mu
\]
and it is an easy task in seeing that the function $x\mapsto\int_{V}M\left(
(Fa_{j}(v)(\chi^{\ast}\cdot\nabla_{x}\phi)\right)  d\mu$ is a compactly
supported $\mathcal{C}^{1}$ function, so that
\begin{align*}
I_{1} &  =\sum_{i,j=1}^{d}\int_{\mathbb{R}_{T}^{d}}\rho_{0}dxdt\left(
\int_{V}\frac{\partial}{\partial x_{j}}M\left(  (Fa_{j}(v)\chi_{i}^{\ast}%
\frac{\partial\phi}{\partial x_{i}}\right)  d\mu\right)  \\
&  =-\sum_{i,j=1}^{d}\left\langle \left(  \int_{V}M\left(  (Fa_{j}%
(v)(\chi^{\ast}\cdot\nabla_{x}\phi)\right)  d\mu\right)  ,\frac{\partial
\rho_{0}}{\partial x_{j}}\right\rangle .
\end{align*}
So we define the matrix $D(x)=(d_{ij})_{1\leq i,j\leq d}$ by $d_{ij}=\int
_{V}M\left(  \chi_{i}^{\ast}a_{j}(v)F\right)  d\mu$, that is,%
\begin{equation}
D(x)=\int_{V}M\left(  \chi^{\ast}\otimes(a(v)F)\right)  d\mu\text{ with }%
\chi^{\ast}\otimes(a(v)F)=(\chi_{i}^{\ast}a_{j}(v)F)_{1\leq i,j\leq
d}.\label{1.11}%
\end{equation}
Then
\[
I_{1}=-\left\langle D(x)\nabla\phi,\nabla\rho_{0}\right\rangle =-\left\langle
D(x)^{T}\nabla\rho_{0},\nabla\phi\right\rangle =\left\langle \Div_{x}\left(
D(x)^{T}\nabla\rho_{0}\right)  ,\phi\right\rangle
\]
where $D(x)^{T}$ stands for the transpose of the matrix $D(x)$. It emerges
that (\ref{pbmacro1}) is the variational formulation (in the usual sense of
distributions) of
\begin{equation}
\frac{\partial\rho_{0}}{\partial t}-{\Div}_{x}\left(  D(x)^{T}\nabla_{x}%
\rho_{0}+U(x)\rho_{0}\right)  =0\text{ in }\mathbb{R}_{T}^{d}\label{eqsol}%
\end{equation}
where $D(x)$ and $U(x)$ are defined respectively by (\ref{1.11}) and
(\ref{1.10}).

Let us now pay attention to the initial condition satisfied by $f_{0}$. We
assume without losing the generality that the function $\rho_{0}$ is smooth
enough. Let us consider a similar test function $\psi^{\varepsilon
}(t,x,v)=\phi(t,x)+\varepsilon\varphi(t,x,\frac{x}{\varepsilon},v)$ as defined
in (\ref{testfunc}), but with $\phi(t,x)=\eta(t)\phi(x)$, where $\eta
\in\mathcal{C}^{\infty}([0,T])$ and $\eta(T)=0$. It's already known that there
exists a unique $\chi^{\ast}\in\mathcal{C}_{0}^{\infty}(\mathbb{R}^{d}\times
V;B_{A}^{2}(\mathbb{R}_{y}^{d}))$ satisfying (\ref{chi}), so that
$\psi^{\varepsilon}(t,x,v)=\eta(t)(\phi(x)+\varepsilon\chi^{\ast}(x,\frac
{x}{\varepsilon},v)\cdot\nabla_{x}\phi(x))$. Then
\begin{align}
\int_{\mathbb{R}_{T}^{d}\times V}\frac{\partial f_{\varepsilon}}{\partial
t}\psi^{\varepsilon}dtd\mu(v)dx &  =\int_{\mathbb{R}_{x}^{d}\times V}\left(
\int_{0}^{T}\frac{\partial f_{\varepsilon}}{\partial t}\psi^{\varepsilon
}dt\right)  d\mu(v)dx\label{intini}\\
&  =\int_{\mathbb{R}_{x}^{d}\times V}\left[  \int_{0}^{T}\left(
\frac{\partial}{\partial t}(f_{\varepsilon}\psi^{\varepsilon})-f_{\varepsilon
}\left(  \frac{\partial\psi}{\partial t}\right)  ^{\varepsilon}\right)
dt\right]  dxd\mu(v)\nonumber\\
&  =-\eta(0)\int_{\mathbb{R}_{x}^{d}\times V}(f_{\varepsilon})_{t=0}%
(\phi+\varepsilon(\chi^{\ast})^{\varepsilon}\cdot\nabla_{x}\phi)dxd\mu
(v)\nonumber\\
&  -\int_{\mathbb{R}_{T}^{d}\times V}\eta^{\prime}f_{\varepsilon}%
(\phi+\varepsilon(\chi^{\ast})^{\varepsilon}\cdot\nabla_{x}\phi)dtdxd\mu
(v)\nonumber\\
&  =-\eta(0)\int_{\mathbb{R}_{x}^{d}\times V}f_{\varepsilon}^{in}%
(\phi+\varepsilon(\chi^{\ast})^{\varepsilon}\cdot\nabla_{x}\phi)dxd\mu
(v)\nonumber\\
&  -\int_{\mathbb{R}_{T}^{d}\times V}\eta^{\prime}f_{\varepsilon}%
(\phi+\varepsilon(\chi^{\ast})^{\varepsilon}\cdot\nabla_{x}\phi)dtdxd\mu
(v).\nonumber
\end{align}
But thanks to the first equation of (\ref{EH0}), we have
\begin{equation}
\frac{\partial f_{\varepsilon}}{\partial t}(t,x,v)=-\frac{1}{\varepsilon
}a(v)\cdot\nabla_{x}f_{\varepsilon}(t,x,v)+\frac{1}{\varepsilon^{2}%
}\mathcal{Q}_{\varepsilon}f_{\varepsilon}(t,x,v).\label{int}%
\end{equation}
Multiplying (\ref{int}) by $\psi^{\varepsilon}(t,x,v)=\phi(t,x)+\varepsilon
\varphi(t,\frac{x}{\varepsilon},x,v)=\eta(t)(\phi(x)+\varepsilon\chi^{\ast
}(x,\frac{x}{\varepsilon},v)\cdot\nabla_{x}\phi(x))$ and integrating by parts
on $\mathbb{R}_{T}^{d}\times V$, one has for each term in the right-hand
side,
\begin{align}
&  \int_{\mathbb{R}_{T}^{d}\times V}\psi^{\varepsilon}(a(v)\cdot\nabla
_{x}f_{\varepsilon})d\mu(v)dxdt\label{intdelx}\\
&  =-\int_{\mathbb{R}_{T}^{d}\times V}\eta f_{\varepsilon}[a(v)\cdot\nabla
_{x}\phi+\varepsilon a(v)\cdot(\nabla_{x}\varphi)^{\varepsilon}+a(v)\cdot
(\nabla_{y}\varphi)^{\varepsilon}]d\mu(v)dxdt\nonumber
\end{align}
and (in view of (\ref{intvuLw}), and since $\mathcal{Q}_{\varepsilon}^{\ast
}\phi=0$)
\begin{align}
\int_{\mathbb{R}_{T}^{d}\times V}\psi^{\varepsilon}(\mathcal{Q}_{\varepsilon
}f_{\varepsilon})dtdxd\mu(v) &  =\int_{\mathbb{R}_{T}^{d}\times V}%
f_{\varepsilon}\eta(t)(\mathcal{Q}_{\varepsilon}^{\ast}(\phi+\varepsilon
\varphi^{\varepsilon})dtdxd\mu(v)\label{intLf}\\
&  =\varepsilon\int_{\mathbb{R}_{T}^{d}\times V}f_{\varepsilon}\eta
(t)(\mathcal{Q}_{\varepsilon}^{\ast}\varphi)^{\varepsilon}dtdxd\mu
(v).\nonumber
\end{align}
Taking into account (\ref{intini})-(\ref{intLf}), we obtain
\begin{align*}
&  \int_{\mathbb{R}_{T}^{d}\times V}f_{\varepsilon}\left(  \eta^{\prime}%
(\phi+\varepsilon(\chi^{\ast})^{\varepsilon}\cdot\nabla_{x}\phi)+\eta\left[
a(v)\cdot(\nabla_{x}\varphi)^{\varepsilon}+\frac{1}{\varepsilon}\left\{
(P^{\ast}\varphi)^{\varepsilon}+a(v)\cdot\nabla_{x}\phi\right\}  \right]
\right)  dtdxd\mu(v)\\
&  =\eta(0)\int_{\mathbb{R}_{x}^{d}\times V}f^{0}(\phi+\varepsilon(\chi^{\ast
})^{\varepsilon}\cdot\nabla_{x}\phi)dxd\mu(v).
\end{align*}
Passing to the limit as $E^{\prime}\ni\varepsilon\rightarrow0$, using the fact
that $P^{\ast}(\varphi)=-a(v)\cdot\nabla_{x}\phi$ and $\varphi=\chi^{\ast
}\cdot\nabla_{x}\phi$ we obtain
\begin{equation}
\int_{\mathbb{R}_{T}^{d}\times V}M(f_{0}[\eta^{\prime}\phi+\eta a(v)\cdot
\nabla_{x}(\chi^{\ast}\cdot\nabla_{x}\phi)])dtd\mu(v)dx=\eta(0)\int
_{\mathbb{R}_{x}^{d}\times V}f^{0}(x,v)\phi(x)dxd\mu(v).\label{eqci}%
\end{equation}
Now, integrating by part the left-hand side of (\ref{eqci}) over $t$ and $x$,
using (\ref{KerP}) and (\ref{solution}), we have
\begin{align}
&  \int_{\mathbb{R}_{T}^{d}\times V}M(f_{0}[\eta^{\prime}\phi+\eta
a(v)\cdot\nabla_{x}(\chi^{\ast}\cdot\nabla_{x}\phi)])dtd\mu(v)dx\label{eqci1}%
\\
&  =-\int_{\mathbb{R}_{T}^{d}}\left(  \frac{\partial\rho_{0}}{\partial
t}-\Div_{x}\left(  U(x)\rho_{0}\right)  -\Div_{x}\left(  D(x)^{T}\nabla
_{x}\rho_{0}\right)  \right)  \eta(t)\phi(x)dtdx\nonumber\\
&  +\eta(0)\int_{\mathbb{R}_{x}^{d}}\rho_{0}(0,x)\phi(x)dx,\nonumber
\end{align}
where $U(x)$ and $D(x)$ are defined above. Combining (\ref{eqci}) and
(\ref{eqci1}) in view of (\ref{eqsol}), we are led to
\[
\eta(0)\int_{\mathbb{R}_{x}^{d}}\rho_{0}(0,x)\phi(x)dx=\eta(0)\int
_{\mathbb{R}_{x}^{d}}\left(  \int_{V}f^{0}(x,v)d\mu(v)\right)  \phi(x)dx.
\]
And thanks once more to the equivalence in the sense of distributions, we
obtain
\begin{equation}
\rho_{0}(0,x)=\int_{V}f^{0}(x,v)d\mu(v)\text{ in }\mathbb{R}_{x}%
^{d}.\label{cdin}%
\end{equation}
Combining (\ref{eqsol}) and (\ref{cdin}), we get (\ref{solf0}).

To end this proof, we just notice that, it is a standard process to prove
existence and uniqueness for (\ref{solf0}); see e.g. \cite{Ladyzenskaja}.
\end{proof}

\begin{remark}
\label{r0}\emph{In the proof of Theorem \ref{theosol}, we have used the
condition (\ref{vfc}). However it is worth noticing that (\ref{vfc}) is made
only for simplification of the presentation of the results. It is to ensure
the existence of the solution to (\ref{chi}). If it is not satisfied, we may
replace in (\ref{chi}) the function }$a_{j}(v)$\emph{ (where }$a(v)=(a_{j}%
(v))_{1\leq j\leq d}$\emph{) by }$a_{j}(v)-\int_{V}M(a_{j}(v)F(x,\cdot
,v))d\mu(v)$\emph{, so that the solution of the corresponding equation exists.
As seen in \cite{Goudon-Mellet}, the conclusion of Theorem \ref{theosol} still
holds, but with a slightly change on the coefficients }$D(x)$\emph{ and
}$U(x)$\emph{.}
\end{remark}

\section{Applications of Theorem \ref{theosol} to some physical
situations\label{sec6}}

The proof of Theorem \ref{theosol} has been performed under the abstract
assumption (\ref{H2}). This assumption models in some sense the way the
heterogeneities (microstructures) are distributed in the medium in which the
phenomenon occurs. As we have seen in the statement of Theorem \ref{theosol},
(\ref{H2}) plays a crucial role in determining of the drift and diffusion
operators, in that their coefficients depend on functions having the same
behaviour as the function $y\mapsto\sigma(x,y,v,w)$. In this section we
provide some physical situations in which (\ref{H2}) holds.

\subsection{Periodic homogenization problem \label{subsec5.1}}

We assume in this section that the microstructures are uniformly distributed
in the medium here represented by $\mathbb{R}_{x}^{d}$. This amounts to saying
that the scattering rate function is periodic with respect to the variable
$y$. We may without losing the generality assume that the function
$y\mapsto\sigma(x,y,v,w)$ is periodic of period $1$ in each direction, that
is,
\[
\sigma(x,y+k,v,w)=\sigma(x,y,v,w)\text{ for all }(x,v,w)\in\mathbb{R}%
^{d}\times V\times V\text{, all }k\in\mathbb{Z}^{d}\text{ and a.e. }y\in Y
\]
where $Y=(0,1)^{d}$ and $\mathbb{Z}$ denotes the integers. Then we are led to
(\ref{H2}) with $A=\mathcal{C}_{per}(Y)$, the Banach algebra of continuous
$Y$-periodic functions on $\mathbb{R}^{d}$, which is an algebra with mean
value, the mean value being here determined by $M(u)=\int_{Y}u(y)dy$ for any
$u\in\mathcal{C}_{per}(Y)$. In that case, we have $B_{A}^{p}(\mathbb{R}%
^{d})=L_{per}^{p}(Y)$ ($1\leq p\leq\infty$), the Banach space of functions
$u\in L_{loc}^{p}(\mathbb{R}^{d})$ that are $Y$-periodic. Assumptions
(\ref{H2}) and (\ref{vfc}) become respectively
\begin{equation}
\sigma\in\mathcal{B}(\mathbb{R}^{d}\times V\times V;L_{per}^{\infty
}(Y))\label{5.1}%
\end{equation}
and
\begin{equation}
\iint_{Y\times V}a(v)F(x,\cdot,v)dyd\mu(v)=0\text{ and }F>0\text{
a.e.,}\label{5.2}%
\end{equation}
and Theorem \ref{theosol} therefore takes the following form.

\begin{theorem}
\label{t5.1}Let $A$ be an algebra with mean value on $\mathbb{R}^{d}$. Assume
\emph{(\ref{5.1})} and \emph{(\ref{5.2})} hold. For each $\varepsilon>0$, let
$f_{\varepsilon}$ be the unique solution of \emph{(\ref{EH0})}. Then, there
exists $f_{0}\in L^{2}(\mathbb{R}_{T}^{d}\times V;L_{per}^{2}(Y))$ such that
the sequence $(f_{\varepsilon})_{\varepsilon\in E}$ weakly two-scale converges
in $L^{2}((0,T)\times\mathbb{R}_{x}^{d}\times V)$ towards $f_{0}$. Moreover
$f_{0}$ has the form $f_{0}(t,x,y,v)=F(x,y,v)\rho_{0}(t,x)$ where
$F\in\mathcal{C}^{1}(\mathbb{R}_{x}^{d};L^{2}(V;L_{per}^{2}(Y))\cap\ker P)$ is
given by \emph{(\ref{0.1})} and $\rho_{0}$ is the unique solution of the
Cauchy problem
\[
\left\{
\begin{array}
[c]{l}%
\frac{\partial\rho_{0}}{\partial t}-\Div_{x}\left(  D(x)^{T}\nabla_{x}\rho
_{0}+U(x)\rho_{0}\right)  =0\text{ in }(0,T)\times\mathbb{R}_{x}%
^{d}\ \ \ \ \ \ \ \ \ \ \ \ \ \\
\\
\rho_{0}(0,x)=\int_{V}f^{0}(x,v)d\mu(v)\text{ in }\mathbb{R}_{x}^{d}%
\end{array}
\right.
\]
where
\[
D(x)=\iint_{Y\times V}\chi^{\ast}\otimes(a(v)F)dyd\mu(v)=\left(
\iint_{Y\times V}\chi_{i}^{\ast}a_{j}(v)Fdyd\mu(v)\right)  _{1\leq i,j\leq d}%
\]
with $(\chi^{\ast}\otimes a(v)F)=\left(  \chi_{i}^{\ast}a_{j}(v)F\right)
_{1\leq i,j\leq d}$,%
\[
U(x)=\iint_{Y\times V}\chi^{\ast}a(v)\cdot\nabla_{x}Fdyd\mu(v)=\left(
\iint_{Y\times V}\chi_{i}^{\ast}a(v)\cdot\nabla_{x}Fdyd\mu(v)\right)  _{1\leq
i\leq d}%
\]
and where $\chi^{\ast}$ the unique solution in $\mathcal{C}_{0}^{\infty
}(\mathbb{R}^{d}\times V;L_{per}^{2}(Y))^{d}$ of the corrector problem:
\[
P^{\ast}\chi^{\ast}=-a(v)\text{ and }\iint_{Y\times V}\chi^{\ast}%
(x,\cdot,v)dyd\mu(v)=0,
\]
$P^{\ast}$ being the adjoint operator of $P$.
\end{theorem}

Theorem \ref{t5.1} is nothing else but Theorem 3.11 in \cite{Goudon-Mellet}.

\subsection{Almost periodic homogenization problem \label{subsec5.2}}

We assume in this subsection that the microstructures are almost uniformly
distributed in the medium, which we represent by the function $y\mapsto
\sigma(x,y,v,w)$ is almost periodic in the Besicovitch sense
\cite{Besicovitch}. It is an easy exercise in seeing that the convenient
algebra with mean value we should use here is the Banach algebra
$A=AP(\mathbb{R}^{d})$ of continuous almost periodic functions defined on
$\mathbb{R}^{d}$ \cite{Bohr}. A function $u:\mathbb{R}^{d}\rightarrow
\mathbb{C}$ is said to be continuous almost periodic (or Bohr almost periodic)
if the set of all its translates $\{u(\cdot+a):a\in\mathbb{R}^{d}\}$ is
pre-compact in the $\sup$-norm topology in $\mathbb{R}^{d}$. The Banach
algebra $AP(\mathbb{R}^{d})$ is an algebra with mean value with the property
that the mean value of a function $u\in AP(\mathbb{R}^{d})$ is the unique
constant belonging to the closed convex hull of the family of the translates
$\{u(\cdot+a):a\in\mathbb{R}^{d}\}$; see e.g. \cite{Jacobs}.

With this in mind, we observe that the space $B_{A}^{2}(\mathbb{R}^{d})$ is
actually the well known Besicovitch space of almost periodic functions on
$\mathbb{R}^{d}$ \cite{Besicovitch}. The conclusion of Theorem \ref{theosol}
holds in this case, and it is a generalization of the periodic setting since
any periodic function is also almost periodic.

\subsection{Asymptotic periodic homogenization problem \label{subsec5.3}}

In this subsection we assume that the medium is a periodic one with defects,
that is, the microstructures, are up to local perturbations, uniformly
distributed. This means that the function $y\mapsto\sigma(x,y,v,w)$ can be
obtained as a sum of a periodic function and a function that vanishes at
infinity. The suggested algebra with mean value in this case is therefore
$A=\mathcal{C}_{per}(Y)+\mathcal{C}_{0}(\mathbb{R}^{d})$ where $\mathcal{C}%
_{0}(\mathbb{R}^{d})$ stands for the space of those continuous functions $u$
in $\mathbb{R}^{d}$ that vanish at infinity. Since $\lim_{\left\vert
y\right\vert \rightarrow\infty}u(y)=0$ for any $u\in\mathcal{C}_{0}%
(\mathbb{R}^{d})$, any element in $A$ is asymptotically a periodic function.
The mean value of a function in $A$ is defined as follows: for $A\ni
u=u_{per}+u^{0}$ with $u_{per}\in\mathcal{C}_{per}(Y)$ and $u^{0}%
\in\mathcal{C}_{0}(\mathbb{R}^{d})$, $M(u)=\int_{Y}u_{per}(y)dy$.

\subsection{Asymptotic periodic homogenization problem \label{subsec5.4}}

In Subsection \ref{subsec5.3} above, we may replace $\mathcal{C}_{per}(Y)$ by
$AP(\mathbb{R}^{d})$ and get the algebra of asymptotic almost periodic
functions denoted by $A=AP(\mathbb{R}^{d})+\mathcal{C}_{0}(\mathbb{R}^{d})$.
In this case, we may assume that the function $y\mapsto\sigma(x,y,v,w)$
belongs to $B_{A}^{2}(\mathbb{R}^{d})$. The conclusion of Theorem
\ref{theosol} holds as well.

\end{document}